\documentclass{amsart}
\usepackage[latin1]{inputenc}
\usepackage{ae,graphics}

\theoremstyle{plain}
\newtheorem{lemma}{Lemma}
\newtheorem{theorem}{Theorem}
\newtheorem{corollary}[theorem]{Corollary}

\theoremstyle{definition}
\newtheorem{definition}{Definition}
\newtheorem{remark}{Remark}

\newtheorem{example}{Example}

\DeclareMathOperator{\Fac}{Fac}

\author[J. Cassaigne]{Julien Cassaigne}
\address{J. Cassaigne, Institut de Mat{\'e}matiques de Luminy, CNRS, Campus de Luminy, Case 907, 13288 Marseille Cedex 9, France}
\email{cassaigne@iml.univ-mrs.fr}
\author[E. Duch\^ene.]{Eric Duch\^ene}
\address{E. Duch\^ene, Universit{\'e} de Lyon, CNRS, Universit{\'e} Lyon 1, LIRIS, UMR5205, F-69622, France}
\email{eric.duchene@univ-lyon1.fr}
\author[M. Rigo]{Michel Rigo}
\address{M. Rigo, University of Liege, Department of Mathematics, Grande traverse 12 (B37), B-4000 Li{\`e}ge, Belgium.}
\email{M.Rigo@ulg.ac.be}

\title{Invariant games and non-homogeneous Beatty sequences}


\begin{document}

\begin{abstract}
    We characterize all the pairs of complementary non-homogenous
    Beatty sequences $(A_n)_{n\ge 0}$ and $(B_n)_{n\ge 0}$ for which
    there exists an invariant game having exactly $\{(A_n,B_n)\mid
    n\ge 0\}\cup \{(B_n,A_n)\mid n\ge 0\}$ as set of
    $\mathcal{P}$-positions. Using the notion of Sturmian word and
    tools arising in symbolic dynamics and combinatorics on words,
    this characterization can be translated to a decision procedure
    relying only on a few algebraic tests about algebraicity or
    rational independence.  Given any four real numbers defining the two
    sequences, up to these tests, we can therefore decide whether or
    not such an invariant game exists.
\end{abstract}

\maketitle

\section{Introduction}

This paper presents some interplay between combinatorics on words and
combinatorial game theory.  Given four real numbers
$\alpha,\beta,\gamma,\delta$ defining two sequences $(A_n)_{n\ge 1}$
and $(B_n)_{n\ge 1}$ making a partition of $\mathbb{N}_{>0}$ with
$A_n:=\lfloor n\alpha+\gamma \rfloor$ and $B_n:=\lfloor n\beta+\delta
\rfloor$, for all $n\ge 1$. Is there an invariant game whose set of
$P$-positions is exactly
$\{(A_n,B_n)\mid n\ge 1\}\cup \{(B_n,A_n)\mid n\ge 1\}\cup \{(0,0)\}$?\\

Even though some basic notions about combinatorial game theory will be
recalled in this paper, the reader can refer to \cite{WW} for some
background on the topic.

\subsection{Invariant games} The last two authors introduced the
notion of an {\em invariant game} \cite{DR} as follows. Consider a
two-player impartial removal game played on some piles of tokens (at
least one pile).  The two players play alternatively.  For each move
belonging to the set of allowed moves prescribed by the game, one can
remove a positive number of tokens from one or several piles. The
player who makes the last move wins. Such a game is said to be {\em
  invariant} if all moves are ``playable from any position'', i.e.,
each allowed move is independent from the position it is played from
(provided there is enough tokens left). For instance, in the game of
Nim, a player can remove any positive number of tokens from one pile.
In the game of Wythoff played on $2$ piles \cite{W}, one can moreover remove
the same number of tokens simultaneously from the two piles. For an
example of ``non-invariant'' or {\em variant} game, consider the
following situation. We play with a single pile of tokens. If the
number of tokens is a prime, then one has to remove a unique token.
Otherwise, the player may remove any positive even number of tokens.
Clearly, in that game, the allowed moves depend on the actual
position. Other examples of {\em variant} games are the Rat or Euclid games \cite{Rat,Euclid}.

Recall that a position is said to be a {\em $P$-position}, if there
exists a strategy for the second player (i.e., the player who will
play on the next round) to win the game, whatever the moves of the
first player are.  Note that games having different sets of moves may
share the same set of $P$-positions. In the literature, studies about
removal games on heaps generally fix the rule-set of the game and try
to find nice characterizations of the underlying $P$-positions.
Recently, several works were focused on a reversed vision of it: starting from
a sequence $S$ of tuples of integers, can one find one or several
games having $S$ as set of $P$-positions?  In \cite{dfnr}, are
considered the invariant extensions (adding new moves) and
restrictions (removing some moves) of Wythoff's game that preserve its
set of $P$-positions. The set of $P$-positions of Wythoff's game has a
nice algebraic characterization: the $n$-th $P$-position is $(\lfloor
n\tau\rfloor, \lfloor n\tau^2\rfloor)$ where $\tau$ is the golden
ratio \cite{W}. In \cite{bo}, the same problem is considered for
non-invariant extensions and restrictions leading to the same set of
$P$-positions. In general, obtaining games with the same set of
$P$-positions is challenging. It has been recently proved that there
is no algorithm, given two games played over the same number of heaps
(and both described by a finite set of allowed moves), deciding
whether or not these two games have the same set of $P$-positions
\cite{Lar1}.

\smallskip

In that context, a general question is the following one: Given a pair
$(A_n,B_n)_{n\ge 0}$ of {\em complementary sequences}, i.e.,
$\{A_n\mid n\ge 1\}$ and $\{B_n\mid n\ge 1\}$ make a partition of
$\mathbb{N}_{>0}$ and $A_0=B_0=0$, {\em does there exist an invariant
  game having $$\mathcal{P}:=\{(A_n,B_n)\mid n\ge 0\}\cup
  \{(B_n,A_n)\mid n\ge 0\}$$ as set of $P$-positions?}  To tackle such
a problem, one can specialize the question to some specific instance
or to well-known families of sequences.  \smallskip

In \cite{DR}, such invariant games are obtained for some particular
family of homogenous Beatty sequences, namely when $A_n=\lfloor
n\alpha\rfloor$, for all $n\ge 0$, where $\alpha$ is a quadratic
irrational number having an ultimately periodic continued fraction
expansion of the kind $(1;\overline{1,k})$. Note that for $k=1$, we
find back the golden ratio with periodic continued fraction expansion
$(1;\overline{1})$. Games related to quadratic irrational numbers with
expansion $(1;\overline{k})$ have been considered in \cite{W3}.

Moreover, the following conjecture was stated:
for any pair $(A_n,B_n)_{n\ge 0}$ of complementary (homogenous) Beatty
sequences, i.e., for two irrational numbers $\alpha,\beta>0$ such that
$\frac{1}{\alpha}+\frac{1}{\beta}=1$, $A_n=\lfloor n\alpha\rfloor$ and
$B_n=\lfloor n\beta\rfloor$, there exists an invariant game having
$\mathcal{P}$ as set of $P$-positions.

\subsection{Superadditivity} In \cite{Lar}, Larsson {\em et al.} give sufficient
conditions on a pair $(A_n,B_n)_{n\ge 0}$ of complementary sequences
that guarantee the existence of an invariant game with the prescribed
set of $P$-positions. Their main condition is that the sequence
$(B_n)_{n\ge 0}$ is {\em $B_1$-superadditive}, i.e., for all $m,n>0$,
$B_m+B_n\le B_{m+n}<B_m+B_n+B_1$. They observe that any pair of
complementary homogenous Beatty sequences satisfy in particular this
latter property. Therefore, their result provides a positive answer
(even for a wider family of complementary sequences) to the question of
the existence of an invariant game as conjectured in \cite{DR}.

The problem being solved positively for homogenous Beatty sequences,
it is natural to address the same question in the framework of
complementary non-homogenous Beatty sequences.

Note that the authors
of \cite{Lar} have conjectured that a pair $(A_n,B_n)_{n\ge 1}$ of
non-homogeneous complementary Beatty sequences with $A_1=1$ gives rise
to an invariant game if and only if the sequence $(B_n)_{>0}$ is
$B_1$-superadditive. We show in Section~\ref{sec:b1}, that this
conjecture turns out to be false.

\subsection{The case of non-homogeneous Beatty sequences}
In this paper, we consider any pair $(A_n,B_n)_{n\ge 0}$ of complementary
{\em non-homogeneous Beatty sequences} defined as follows. Let $\alpha<\beta$
be positive irrational numbers satisfying
\begin{equation}
    \label{eq:r1}
    \frac{1}{\alpha}+\frac{1}{\beta}=1
\end{equation}
In particular, this implies that $1<\alpha<2<\beta$.  Let
$\gamma,\delta$ be two real numbers.  We set $A_0=B_0=0$ and for $n\ge
1$, $$A_n:=\lfloor n\alpha+\gamma \rfloor,\quad B_n:=\lfloor
n\beta+\delta \rfloor.$$ In what follows, we assume that
\begin{equation}
    \label{eq:r2}
    A_1=1 \text{ and  }B_1\geq 3,
\end{equation}
\begin{equation}
    \label{eq:r3}
    \{A_n\mid n\ge 1\} \text{ and  }\{B_n\mid n\ge 1\}\text{ make a partition of }\mathbb{N}_{>0}.
\end{equation}

\begin{remark}
    Condition \ref{eq:r3} is standard in the context of removal games on heaps. Indeed, it allows all the moves of the type $(0,k)_{k>0}$, which means that the games we are looking for are extensions of Nim. In addition, we have set $A_1=1$ to guarantee the inequality $A_n<B_n$ for all $n>0$. However this condition is not crucial and its removal would not fundamentally change our work.
The condition $B_1\geq 3$ is thus the only limitation of our result.
\end{remark}

Note that $A_1=1$ implies $1-\alpha\le\gamma<2-\alpha$.  Necessary and
sufficient conditions for a non-homogeneous Beatty sequence to satisfy
the above partition constraint \eqref{eq:r3} are given by Fraenkel in the
following result:
\begin{theorem}\label{thm1}\cite{Fra}
    Let $\alpha<\beta$ be positive irrational numbers satisfying
    $\frac{1}{\alpha}+\frac{1}{\beta}=1$. Then $(\lfloor
    n\alpha+\gamma \rfloor)_{n>0}$ and $(\lfloor n\beta+\delta
    \rfloor)_{n>0}$ make a partition of $\mathbb{N}_{>0}$ if and only
    if
    \begin{equation}
        \label{eq:r4}
        \frac{\gamma}{\alpha}+\frac{\delta}{\beta}=0 \text{ and, }
    \end{equation}
    \begin{equation}
        \label{eq:r5}
        \text{for all } n\ge 1,\ n\beta+\delta\not\in\mathbb{Z}.
    \end{equation}
\end{theorem}

\begin{remark}\label{rem2}
    Note that the above condition \eqref{eq:r4} and our first
    assumptions on $\alpha$ and $\beta$ imply the following
    dependencies:
    \begin{itemize}
    \item If $0\leq \gamma<2-\alpha$, then we have $\delta\leq 0$.
    \item If $1-\alpha<\gamma<0$, then we get $0<\delta<1$.
    \end{itemize}
\end{remark}

In this paper, we will always assume that $\alpha,\beta,\gamma,\delta$
are real numbers satisfying \eqref{eq:r1}, \eqref{eq:r2}, \eqref{eq:r4}, \eqref{eq:r5} (and consequently \eqref{eq:r3} holds as a consequence of Theorem~\ref{thm1}). We set
$$\mathcal{P}:=\{(A_n,B_n)\mid n\ge 0\}\cup\{(B_n,A_n)\mid n\ge 0\}.$$
We characterize the values of $\alpha,\beta,\gamma,\delta$ for which
there exists an invariant game having $\mathcal{P}$ as set of
$P$-positions. More precisely, this paper is organized as follows. The
first difference $$(\lfloor (n+1)\lambda+\rho\rfloor-\lfloor
n\lambda+\rho\rfloor)_{n\ge 0}$$ of a Beatty sequence associated with
the irrational number $\lambda$ is what is usually called a Sturmian
word. In Section~\ref{sec:2}, we recap the needed background in
combinatorics on words about Sturmian words.  In particular, we are
interested in occurrences of some prescribed factor appearing in such a
word. Similar interplay between combinatorial game theory and combinatorics on words can be found in \cite{DR2,Flora}.

Section~\ref{sec:3} contains the technical core of the paper.  We
provide necessary and sufficient conditions for the existence of an
invariant game having $\mathcal{P}$ as set of $P$-positions. Our
condition has a combinatorial flavor: we define an infinite word
$\mathbf{w}$ built upon $\alpha,\beta,\gamma,\delta$ and the condition
is described in terms of factors occurring or not in this word. If the
$4$-tuple $(\alpha,\beta,\gamma,\delta)$ satisfies this condition,
then this $4$-tuple is said to be {\em good}. In Theorem~\ref{the:1}, we
prove that there is no invariant game having $\mathcal{P}$ as set of
$P$-positions and associated with a $4$-tuple which is not good. In
Theorem~\ref{the:principal}, we show that whenever the $4$-tuple is
good, then there is an invariant game having $\mathcal{P}-\mathcal{P}$
as set of moves.

In Section~\ref{sec:4}, we translate these combinatorial conditions into an
algebraic setting better suited to tests. It turns out that if
$\alpha,\beta,1$ are rationally independent, then we make use of the
two-dimensional version of the density theorem of Kronecker.
Otherwise, one has to study the relative position of a rectangle and a
straight line over the two-dimensional torus
$\mathbb{R}^2/\mathbb{Z}^2$.

In Section~\ref{sec:b1}, answering a conjecture of Larsson, Hegarty
and Fraenkel \cite{Lar}, we show that $B_1$-superadditivity is not a
necessary condition to get an invariant game.

\section{Sturmian sequences for game combinatorists}\label{sec:2}

In this section, we collect the main facts and notions on Sturmian
sequences that will be used in this paper. This section has been
written for a reader having no particular knowledge in combinatorics
on words. For general references, see for instance \cite{cant,lot2}.

An {\em alphabet} is a finite set. An {\em infinite word} or a {\em
  sequence} over the alphabet $\mathcal{A}$ is a map from $\mathbb{N}$
to $\mathcal{A}$. Infinite words will be denoted using bold face
symbols. Note also that the first element of an infinite word has
index $0$.  The so-called Sturmian words form a well-known and
extensively studied class of infinite words over a $2$-letter
alphabet. They can be defined in several equivalent ways, one of them
arising in the context of non-homogenous Beatty sequences. We now
present three equivalent definitions of Sturmian words namely, as
mechanical words, as codings of some rotations and in relation with a
balance property \cite[Chap.~2]{lot2}. Let $\lambda,\rho$ be two real numbers with
$\lambda>0$. In the literature, one usually consider
$\lambda\in(0,1)$ which is not a true restriction but, in that case,
the obtained word is written over the alphabet $\{0,1\}$. For an
arbitrary $\lambda$, the reader will notice that we will have to
consider in some situations its fractional part $\{\lambda\}$.

\subsection{Mechanical words} We define the infinite word
$\mathbf{s}_{\lambda,\rho}=(s_{\lambda,\rho}(n))_{n\ge 0}$ by
$$s_{\lambda,\rho}(n)=\lfloor (n+1)\lambda+\rho\rfloor-\lfloor n\lambda+\rho\rfloor,\quad \forall n\ge 0.$$
This word is often referred as a {\em lower mechanical word}
\cite{lot2}. It is not difficult to see that $s_{\lambda,\rho}(n)$
takes exactly the two values $\lfloor\lambda\rfloor$ and
$\lfloor\lambda\rfloor+1$.

\begin{example}\label{exa:fib}
    Take $\lambda=\rho=1/\tau$ where $\tau$ is the golden ratio
    $(1+\sqrt{5})/2$. The word $\mathbf{s}_{1/\tau,1/\tau}$ is the consecrated Fibonacci word whose first elements are
$$\mathbf{s}_{1/\tau,1/\tau}=101101011011010110101\cdots.$$
\end{example}

\subsection{Coding of rotations} Consider the one-dimensional torus
$\mathbb{T}^1=\mathbb{R}/\mathbb{Z}$ identified with $[0,1)$. Take the map
$R_\lambda:\mathbb{T}^1\to \mathbb{T}^1, x\mapsto \{x+\lambda\}$. One can
study the orbit of $\rho$ (reduced modulo $1$) under the action of
$R_\lambda$. We thus define two intervals $I_{\lfloor\lambda\rfloor}=[0,1-\{\lambda\})$ and
$I_{\lfloor\lambda\rfloor+1}=[1-\{\lambda\},1)$ partitioning $\mathbb{T}^1$.  One can show that
this setting provides another way to define the word
$\mathbf{s}_{\lambda,\rho}$. For all $n\ge 0$, we have
$$s_{\lambda,\rho}(n)=\lfloor\lambda\rfloor \Leftrightarrow R_\lambda^n(\rho)\in I_{\lfloor \lambda\rfloor}\text{ and }s_{\lambda,\rho}(n)=\lfloor\lambda\rfloor+1 \Leftrightarrow R_\lambda^n(\rho)\in I_{\lfloor\lambda\rfloor+1}.$$
This formalism is convenient to describe factors occurring in the
Sturmian word $\mathbf{s}_{\lambda,\rho}$. Recall that a {\em factor}
of length $\ell$ in a word $\mathbf{w}=w_0w_1\cdots$ is a finite
sequence made of $\ell$ consecutive elements: $w_i\cdots
w_{i+\ell-1}$. We say that the factor {\em occurs} in $\mathbf{w}$ {\em in
position} $i$. The set of all factors of length $\ell$ occurring in
$\mathbf{w}$ is denoted by $\Fac_{\mathbf{w}}(\ell)$ and the whole set of
factors of $\mathbf{w}$ is $$\Fac_{\mathbf{w}}=\bigcup_{\ell\ge 0}
\Fac_{\mathbf{w}}(\ell).$$ For a binary word $v=v_0v_1\cdots v_m$, for all $t$, $v_t\in\{\lfloor\lambda\rfloor,\lfloor\lambda\rfloor+1\}$, we
define a half-interval $I_{v,\lambda}$ of $\mathbb{T}^1$ as
\begin{equation}
    \label{eq:Iv}
I_{v,\lambda}:=I_{v_0}\cap R^{-1}_{\lambda}(I_{v_1})\cap\dots\cap R^{-m}_{\lambda}(I_{v_m}).
\end{equation}
One can show that $v$ occurs in $\mathbf{s}_{\lambda,\rho}$ in
position $i$ if and only if $R^i_{\lambda}(\rho) \in I_{v,\lambda}$. See
\cite[Section 2.1.2]{lot2}.

\subsection{A balance property} Let $\mathcal{A}$ be a finite
alphabet. Let $a\in\mathcal{A}$ and $u$ be a finite word over
$\mathcal{A}$. We denote by $|u|$ the length of $u$ and by $|u|_a$ the
number of occurrences of $a$ in $u$. An infinite word $\mathbf{w}$
over $\mathcal{A}$ is said to be {\em balanced} if, for all $n\ge 0$,
all $u,v\in \Fac_{\mathbf{w}}(n)$ and all $a\in\mathcal{A}$, we have
$||u|_a-|v|_a|\le 1$.

\begin{theorem}\cite[Theorem 2.1.5]{lot2}
    An infinite word over $\{0,1\}$ is Sturmian if and only if it is aperiodic and
    balanced.
\end{theorem}

This result implies that, for a given Sturmian word
$\mathbf{s}_{\lambda,\rho}$, up to permutation of the letters (i.e., up to abelian equivalence), there
are exactly two kinds of factors of length $\ell$, those having either
$\lceil \ell\{\lambda\} \rceil$, or $\lceil \ell\{\lambda\} \rceil-1$, symbols
$\lfloor\lambda\rfloor+1$.  The corresponding factors will be called
respectively {\em heavy} and {\em light}. In \cite{pav}, the following intervals are defined, for all $\ell>0$,
\begin{equation}
    \label{eq:IH}
    I_{H,\lambda}(\ell) = [1-\{\ell\lambda\},1) \text{ and }
        I_{L,\lambda}(\ell) = [0,1-\{\ell\lambda\})
\end{equation}
	 and it is
        proved that the factor of length $\ell$ occurring in position
        $i$ in $\mathbf{s}_{\lambda,\rho}$ is heavy if and only if
        $R^i_{\lambda}(\rho) \in I_{H,\lambda}(\ell)$.

        \begin{example}
            For the Fibonacci word introduced in
            Example~\ref{exa:fib}, there are exactly $6$ factors of
            length $5$. Five are light: $10110$, $01101$, $11010$,
            $10101$ and $01011$. They contain $\lceil
            5\{\tau\}\rceil-1=3$ symbols $1$. The unique heavy factor
            of length $5$ is $11011$ with $4$ symbols $1$.
        \end{example}

We will often make use of the following observation.

\begin{remark}\label{rem:heavy-light} Let $u=u_1\cdots u_n$ and $v=v_1\cdots v_n$ be two
    factors of length $n$ occurring in the Sturmian word
    $\mathbf{s}_{\lambda,\rho}$ over the alphabet
    $\{\lfloor\lambda\rfloor,\lfloor\lambda\rfloor+1\}$. We have
$$\sum_{i=1}^nu_i-\sum_{i=1}^nv_i=\left\{
    \begin{array}{rl}
1&\text{if }u\text{ is heavy and }v\text{ is light,}\cr
-1&\text{if }u\text{ is light and }v\text{ is heavy,}\cr
0&\text{if }u,v\text{ are both light (resp. heavy).}\cr
    \end{array}\right.$$
In particular, we have
$$   \sum_{i=1}^nv_i-1  \le    \sum_{i=1}^nu_i \le \sum_{i=1}^nv_i+1.$$
\end{remark}

\subsection{Direct product and synchronization} Let
$\mathbf{s}=(s_{n})_{n\ge 0}$ and $\mathbf{t}=(t_{n})_{n\ge 0}$ be two
infinite words over the alphabets $\mathcal{A}$ and $\mathcal{B}$
respectively. The {\em direct product} of $\mathbf{s}$ and
$\mathbf{t}$ is the sequence $\mathbf{s}\otimes \mathbf{t}=(p_n)_{n\ge
  0}$ where $$p_n=(s_n,t_n),\quad \forall n\ge 0.$$ Observe that
$\mathbf{s}\otimes \mathbf{t}$ is an infinite word over the alphabet
$\mathcal{A}\times\mathcal{B}$ of pairs of symbols.  We denote by
$\pi_1$ and $\pi_2$ the two homomorphisms of projection defined by
$\pi_1(x,y)=x$, $\pi_2(x,y)=y$, for all
$(x,y)\in\mathcal{A}\times\mathcal{B}$, and extended to
$\pi_1(\mathbf{s}\otimes \mathbf{t})=\mathbf{s}$ and
$\pi_2(\mathbf{s}\otimes \mathbf{t})=\mathbf{t}$.  If
$u\in\mathcal{A}^*$ and $v\in\mathcal{B}^*$ are two finite words of
the same length, one can define accordingly $u\otimes v$ and
$\pi_1(u\otimes v)=u$, $\pi_2(u\otimes v)=v$.  For more on
two-dimensional generalization of Sturmian sequences, see for instance
\cite{BV}. For some recurrence properties of direct product, see
\cite{pavel}.

In the next sections, up to some minor modifications of the first
symbol, we will be interested in the direct product of two Sturmian
words $\mathbf{s}_{\lambda,\rho}$ and $\mathbf{s}_{\mu,\nu}$ over the
alphabets
$\mathcal{A}_\lambda=\{\lfloor\lambda\rfloor,\lfloor\lambda\rfloor+1\}$ and
$\mathcal{A}_\mu=\{\lfloor\mu\rfloor,\lfloor\mu\rfloor+1\}$ respectively.

Consider the two-dimensional torus $\mathbb{T}^2=\mathbb{R}^2/\mathbb{Z}^2$ identified with
$[0,1)\times[0,1)$ and the map
$R_{\lambda,\mu}:\mathbb{T}^2\to\mathbb{T}^2,(x,y)\mapsto
(\{x+\lambda\},\{y+\mu\})$. It is obvious that $\mathbb{T}^2$ is split
into four regions of the kind $I_a\times I_b$ where $a\in\mathcal{A}_\lambda$
and $b\in\mathcal{A}_\mu$ in such a way that
$$(\mathbf{s}_{\lambda,\rho}\otimes\mathbf{s}_{\mu,\nu})(n)=(a,b)\Leftrightarrow R_{\lambda,\mu}^n(\rho,\nu)\in I_a\times I_b.$$

\begin{example}\label{exa1}
    Take $\beta=3.99+\sqrt{5}/2\simeq 5.108$ and $\gamma=-0.2$. From
    \eqref{eq:r1} and \eqref{eq:r4}, we get
    $\alpha=\beta/(\beta-1)\simeq 1.243$ and
    $\delta=-\beta\gamma/\alpha\simeq 0.821$. We have
    $\mathcal{A}_\alpha=\{1,2\}$ and $\mathcal{A}_\beta=\{5,6\}$. In
    Figure~\ref{fig:t2}, the starting point $(\{\gamma\},\{\delta\})$
    is denoted by $0$ and the arrow represents the application of
    $R_{\alpha,\beta}$, i.e., a translation of
    $(\{\alpha\},\{\beta\})$ in $\mathbb{T}^2$. The first ten iterations of this map are
    labeled in the figure. The torus is split into the four regions
    corresponding to an occurrence of $(1,5)$, $(1,6)$, $(2,5)$ and
    $(2,6)$ respectively. These regions are denoted respectively $a$,
    $b$, $c$ and $d$.
    \begin{figure}[h!tbp]
          \begin{center}
      \includegraphics{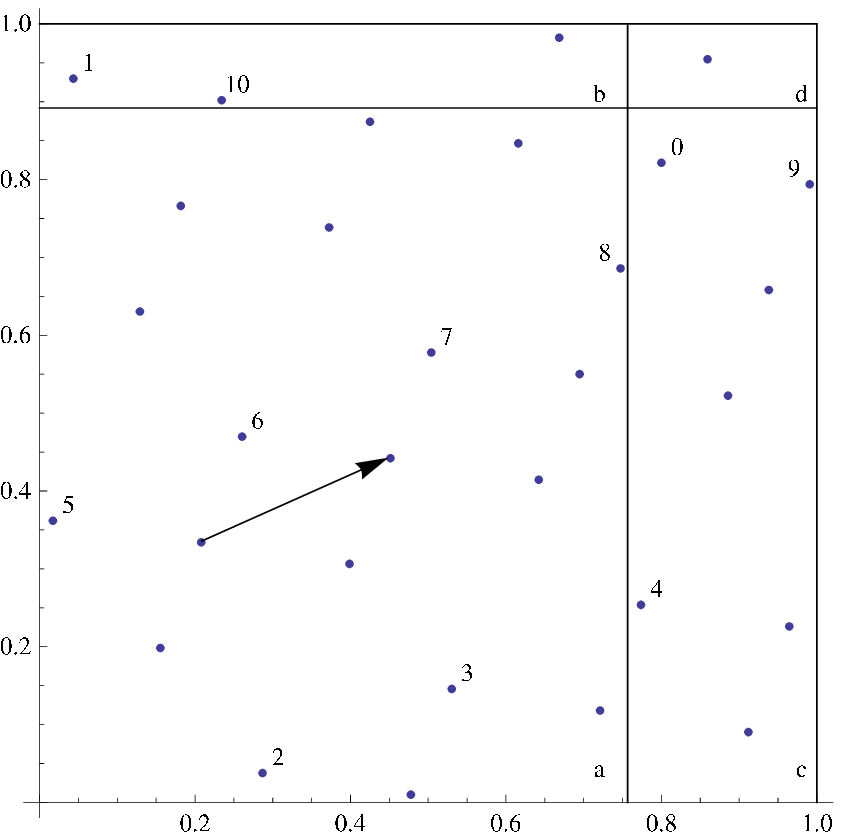}
  \end{center}
        \caption{The torus $\mathbb{T}^2$ split into four regions.}
        \label{fig:t2}
    \end{figure}
    One can notice that the visited regions are: $cbaacaaaacb\cdots$.
    From this, we get
    $\pi_1(\mathbf{s}_{\alpha,\gamma})=21112111121\cdots$ and
      $\pi_2(\mathbf{s}_{\beta,\delta})=565555555556\cdots$.
\end{example}

The next statement is a direct reformulation of~\eqref{eq:Iv}.

\begin{lemma}\label{lem:algo}
    Let $v\in\mathcal{A}_\lambda^+$ and $w\in\mathcal{A}_\mu^+$ be two
  words of the same length. The factor $v\otimes w$ occurs in
  $\mathbf{s}_{\lambda,\rho}\otimes\mathbf{s}_{\mu,\nu}$ in position
  $i$, i.e., $v$ occurs in
  $\mathbf{s}_{\lambda,\rho}$ in position
  $i$ and simultaneously $w$ occurs in
  $\mathbf{s}_{\mu,\nu}$ in position
  $i$, if and only if $R_{\lambda,\mu}^i(\rho,\nu)\in
  I_{v,\lambda}\times I_{w,\mu}$.
\end{lemma}

Let $\ell\ge 1$. Using~\eqref{eq:IH}, again $\mathbb{T}^2$ is split into four regions of
the kind $I_{A,\lambda}(\ell)\times I_{B,\mu}(\ell)$ where
$A,B\in\{L,H\}$ in such a way that the factor of length $\ell$
occurring in $\mathbf{s}_{\lambda,\rho}$ is light and, simultaneously,
the factor of length $\ell$ occurring in $\mathbf{s}_{\mu,\nu}$ is
light, if and only if $R_{\lambda,\mu}^i(\rho,\nu)\in
I_{L,\lambda}(\ell)\times I_{L,\mu}(\ell)$. The other combinations
light/heavy, heavy/light and heavy/heavy are derived accordingly.

\section{A characterization of the $4$-tuples leading to an invariant game}\label{sec:3}

From $\alpha,\beta,\gamma,\delta$, let us define an infinite sequence
$\mathbf{w}=(w_n)_{n\ge 0}$ of pairs taking values in a finite set.
Except maybe for the first symbol, this sequence is the direct product
$\mathbf{s}_{\alpha,\gamma}\otimes\mathbf{s}_{\beta,\delta}$ of two
Sturmian words. Indeed, we have set $A_0=0$ (resp. $B_0=0$) which may
differ from $\lfloor\gamma\rfloor$ (resp. $\lfloor\delta\rfloor$).

For all $n\ge 0$, we set $w_n=(A_{n+1}-A_n,B_{n+1}-B_n)$, that is
$$\mathbf{w}=w_0w_1w_2\cdots=(A_1,B_1)\, (s_{\alpha,\gamma}(1),s_{\beta,\delta}(1))\, (s_{\alpha,\gamma}(2),s_{\beta,\delta}(2))\cdots$$
The purpose of this word $\mathbf{w}$ is the following one. Let $i\le j$. Observe that
\begin{equation}
    \label{eq:sum}
    \sum_{k=i}^j w_i=(A_{j+1}-A_i,B_{j+1}-B_i).
\end{equation}
Referring to the sign of $\gamma$, this word will be denoted by
$\mathbf{w}_+$ (resp. $\mathbf{w}_-$) whenever $\gamma$ is positive
(resp. negative).  Recall that for all $n\ge 0$,
$s_{\alpha,\gamma}(n)\in\{1,2\}$ and $s_{\beta,\delta}(n)\in \{\lfloor
\beta\rfloor, \lfloor \beta\rfloor +1\}$. Since $A_1-A_0=1$, the
following observation is straightforward.

\begin{lemma}\label{lem:ecarts}
For all $n\ge 0$, we have $A_{n+1}-A_{n}\in\{1,2\}$. For all $n\ge 1$, we have
$B_{n+1}-B_{n}\in\{\lfloor\beta\rfloor,\lfloor\beta\rfloor+1\}$
\end{lemma}

\begin{remark}
    Assume that $\gamma<0$. Hence $\lfloor \gamma\rfloor\le -1$ and
    from Remark~\ref{rem2}, $0<\delta<1$, i.e., $\lfloor
    \delta\rfloor=0$. Then $A_1<s_{\alpha,\gamma}(0)=A_1-\lfloor
    \gamma\rfloor$ and $B_1=B_1-\lfloor
    \delta\rfloor=s_{\beta,\delta}(0)$, i.e.,
    $\pi_2(\mathbf{w}_-)=\mathbf{s}_{\beta,\delta}$. In this latter
    case, since $A_1=1$, the word $\mathbf{w}_-$ is defined over the
    alphabet of size $4$
$$\mathcal{A}:=\{(1,\lfloor\beta\rfloor), (1,\lfloor\beta\rfloor+1),
(2,\lfloor\beta\rfloor), (2,\lfloor\beta\rfloor+1)\}$$ which, we
consider to be in one-to-one correspondence with the alphabet
$\{a,b,c,d\}$. We choose $a,b,c,d$ such that $\pi_1(a)=\pi_1(b)=1$,
$\pi_1(c)=\pi_1(d)=2$, $\pi_2(a)=\pi_2(c)=\lfloor\beta\rfloor$ and
$\pi_2(b)=\pi_2(d)=\lfloor\beta\rfloor+1$.
\end{remark}

\begin{example}[For $\gamma<0$]
   We take the same values as in Example~\ref{exa1} and we get the following table.
$${\small
\begin{array}{r|cccccccccccccccccccc}
 & 0 & 1 & 2&3&4&5&6&7&8&9&10&11&12&13&14&15&16&17&18\cr
\hline
A_{n+1}-A_n & 1& 1& 1& 1& 2& 1& 1& 1& 1& 2& 1& 1& 1& 2& 1& 1& 1& 2& 1\cr
B_{n+1}-B_n & 5& 6& 5& 5& 5& 5& 5& 5& 5& 5& 6& 5& 5& 5& 5& 5& 5& 5& 5\cr
\hline
\mathbf{w}_- &a& b& a& a& c& a& a& a& a& c& b& a& a& c& a& a& a& c& a\cr
\end{array}
}$$ Note that $s_{\alpha,\gamma}(0)=2\neq 1$ meaning that the first
sequence in the above table disagrees with
$\mathbf{s}_{\alpha,\gamma}$ for the first term only.  Also, the first
occurrence of $(2,6)=d$ is for $n=30$. Again, $\mathbf{w}_+$ agrees
with the coding of the trajectory described in Example~\ref{exa1},
except for the first symbol.
\end{example}

\begin{remark}
    Assume that $\gamma>0$. In our setting, recall that $\gamma<1$.
    Hence $\lfloor \gamma\rfloor=0$ and we get that
    $A_1=s_{\alpha,\gamma}(0)$, i.e.,
    $\pi_1(\mathbf{w}_+)=\mathbf{s}_{\alpha,\gamma}$. Moreover,
    $\delta<0$ and therefore
    $B_1=\lfloor\beta+\delta\rfloor<s_{\beta,\delta}(0)=\lfloor\beta+\delta\rfloor-\lfloor\delta\rfloor$.
    The word $w_1w_2\cdots$ is thus defined over the alphabet
    $\mathcal{A}$ and $w_0=(1,B_1)$ with $B_1\leq
    \lfloor\beta\rfloor$. So, if $B_1<\lfloor\beta\rfloor$, then
    $\mathbf{w}_+$ is written over an alphabet of size $5$, the symbol
    $(1,B_1)$ appearing only once and in first position (it will be
    denoted by $e$).
\end{remark}

\begin{example}[For $\gamma>0$]
    Take $\beta=4.99+\sqrt{5}/2\simeq 6.108$ and
    $\delta=-1-\sqrt{2}\simeq -2.414$. From \eqref{eq:r1} and
    \eqref{eq:r4}, we get $\alpha=1/(1-1/\beta)\simeq 1.196$ and
    $\gamma=-\alpha\delta/\beta\simeq 0.473$.
$${\small
\begin{array}{r|cccccccccccccccccccc}
 & 0 & 1 & 2&3&4&5&6&7&8&9&10&11&12&13&14&15&16&17&18\cr
\hline
A_{n+1}-A_n &  1 &  1 &  2 &  1 &  1 &  1 &  1 &  2 &  1 &  1 &  1 &  1 &  2 &  1 &  1 &  1 &  1 &  1 &  2 \cr
B_{n+1}-B_n &  3 &  6 &  6 &  7 &  6 &  6 &  6 &  6 &  6 &  6 &  6 &  6 &  6 &  7 &  6 &  6 &  6 &  6 &  6 \cr
\hline
\mathbf{w}_+ &e&a&c&b&a&a&a&c&a&a&a&a&c&b&a&a&a&a&c\\
\end{array}
}$$
Note that $s_{\beta,\delta}(0)=6\neq 3$  meaning that the second sequence in the above table disagrees with $\mathbf{s}_{\beta,\delta}$ for the first term only. Also, the first occurrence of $(2,7)=d$ is for $n=162$.
\end{example}

\begin{example}[For $\gamma>0$]
    Take $\beta=4.99+\sqrt{5}/2\simeq 6.108$ and $\delta=-0.05$. Hence
    $B_1-B_0=\lfloor\beta+\delta\rfloor=6$. Even if $B_1-B_0$ belong
    to $\{\lfloor\beta\rfloor,\lfloor\beta\rfloor+1\}$, this example
    shows that $B_1-B_0\neq
    s_{\beta,\delta}(0)=\lfloor\beta+\delta\rfloor-\lfloor\delta\rfloor=7$.
\end{example}

For the next definition, it is convenient to introduce the {\em Parikh
  vector} of a finite word $w\in\{a,b,c,d\}$. It is defined as
$\Psi(w)=(|w|_a,|w|_b,|w|_c,|w|_d)\in\mathbb{N}^4$.

\begin{definition}\label{def:pbad}
    Let $p=w_0u_1\cdots u_n=w_0u$ be a prefix of length $n+1$
    occurring in $\mathbf{w}_+$.  We say that $p$ is {\em bad} if $u\in\{a,b\}^*$,
    $u_1\cdots u_{n-1}$ is a palindrome and if any of the following
    three situations occur:
     \begin{itemize}
       \item[(i)] $2B_1=\pi_2(u_n)-1$ and there exists $v\in \Fac_{n}(\mathbf{w}_+)$ such that
$|u|_a=|v|_a$, $|u|_b-1=|v|_b$, $|v|_c=1$ or $|u|_a+1=|v|_a$, $|u|_b-2=|v|_b$, $|v|_d=1$, i.e.,
$$\Psi(v)=\Psi(u)+(0,-1,1,0)\text{ or }\Psi(v)=\Psi(u)+(1,-2,0,1)$$
\item[(ii)]  $2B_1=\pi_2(u_n)$ and there exists $v\in \Fac_{n}(\mathbf{w}_+)$ such that
$|u|_a-1=|v|_a$, $|u|_b=|v|_b$, $|v|_c=1$ or $|u|_a=|v|_a$, $|u|_b-1=|v|_b$, $|v|_d=1$, i.e.,
$$\Psi(v)=\Psi(u)+(-1,0,1,0)\text{ or }\Psi(v)=\Psi(u)+(0,-1,0,1)$$
\item[(iii)] $2B_1=\pi_2(u_n)+1$ and there exists $v\in \Fac_{n}(\mathbf{w}_+)$ such that
$|u|_a-1=|v|_a$, $|u|_b=|v|_b$, $|v|_d=1$, or $|u|_a-2=|v|_a$, $|u|_b+1=|v|_b$, $|v|_c=1$, i.e.,
$$\Psi(v)=\Psi(u)+(-1,0,0,1)\text{ or }\Psi(v)=\Psi(u)+(-2,1,1,0).$$
\end{itemize}
\end{definition}

\begin{remark}\label{rem:bad+}
    The latter definition seems rather artificial. But the motivation
    comes from summing up the symbols occurring in such factors. Let $p=w_0u_1\cdots u_n=w_0u$ be a prefix of length $n+1$
    occurring in $\mathbf{w}_+$. First note that
$$\sum_{i=1}^{n}u_i=(A_n-A_1+\pi_1(u_n),B_n-B_1+\pi_2(u_n))$$
Assume that we are in the situation (i) described in Definition~\ref{def:pbad}, then
$$\sum_{i=1}^{n}v_i=\sum_{i=1}^{n}u_i+(1,-1)=(A_n+\pi_1(u_n),B_n+B_1)=(A_{n+1},B_n+B_1).$$
Assume that we are in the situation (ii) described in Definition~\ref{def:pbad}, then
$$\sum_{i=1}^{n}v_i=\sum_{i=1}^{n}u_i+(1,0)=(A_n+\pi_1(u_n),B_n+B_1)=(A_{n+1},B_n+B_1).$$
Assume that we are in the situation (iii) described in Definition~\ref{def:pbad}, then
$$\sum_{i=1}^{n}v_i=\sum_{i=1}^{n}u_i+(1,1)=(A_n+\pi_1(u_n),B_n+B_1)=(A_{n+1},B_n+B_1).$$
\end{remark}

\begin{definition}\label{def:pbad2}
    Let $u=u_1\cdots u_n\in\{a,b\}^*$ be a factor of length $n$ occurring in $\mathbf{w}_-$.
    \begin{itemize}
      \item[(B.1)] If there exists $v\in \Fac_n(\mathbf{w}_-)\cap \{a,b\}^*$
        such that $|v|_b< |u|_b$, i.e.,
$$\Psi(v)=\Psi(u)+(j,-j,0,0), \text{ for some }j>0,$$
then we say that $u$ satisfies property (B.1).
      \item[(B.2)] If there exists $v'\in \Fac_{n-1}(\mathbf{w}_-)$ such that either $|u_1\cdots u_{n-1}|_a= |v'|_a$,
        $|u_1\cdots u_{n-1}|_b-1= |v'|_b$ and $|v'|_c=1$, or $|u_1\cdots u_{n-1}|_a+1= |v'|_a$,
        $|u_1\cdots u_{n-1}|_b-2= |v'|_b$ and $|v'|_d=1$, i.e.,
$$\Psi(v')=\Psi(u_1\cdots u_{n-1})+(0,-1,1,0) \text{ or }\Psi(v')=\Psi(u_1\cdots u_{n-1})+(1,-2,0,1),$$
then we say that $u$ satisfies property (B.2).
\end{itemize}
A factor $u=u_1\cdots u_n\in\{a,b\}^*$ of $\mathbf{w}_-$ is {\em suffix-bad}, if
\begin{itemize}
\item for all $j\in\{1,\ldots,n\}$, the
suffixes $u_j\cdots u_n$ all satisfy property (B.1), or
\item for all $j\in\{1,\ldots,n-1\}$, the
suffixes $u_j\cdots u_n$ all satisfy property (B.2).
\end{itemize}
\end{definition}

\begin{example}
    Assume that $\mathbf{w}_-=aabcdaaabac\cdots$. The prefix $aab$ is
    suffix-bad.  Indeed, the factors $a$, $aa$ and $aaa$ occurring in
    $\mathbf{w}_-$ show that the suffixes $b$, $ab$ and $aab$ satisfy
    (B.1).  Note that the factor $aba$ satisfies (B.2), because its
    prefix of length $2$ contains exactly one $a$ and one $b$, but the
    factor $ac$ contains the same number of $a$'s and one $b$ has been
    replaced with $c$.
\end{example}

Let $u\in \Fac(\mathbf{w}_-)\cap\{a,b\}^*$. Note that, if all suffixes of $u$ satisfy (B.1), then
$u$ ends with $b$.

\begin{definition}\label{def:good}
    A $4$-tuple $(\alpha,\beta,\gamma,\delta)$ is {\em good} if
    \begin{itemize}
      \item $\gamma>0$ and for all prefixes $p=w_0w_1\cdots w_k$ of
        $\mathbf{w}_+$ such that $2\le|p|<B_1$, i.e., $1\le k\le
        B_1-2$, with $w_i\in\{a,b\}$ for $i=1,\ldots,k$, $p$ is not
        bad.
      \item $\gamma<0$ and, for any prefix $p\in\{a,b\}^*$ of $\mathbf{w}_-$ of
        length less than $B_1$, $p$ is not suffix-bad.
    \end{itemize}
\end{definition}

\begin{theorem}\label{the:1}
    If the $4$-tuple $(\alpha,\beta,\gamma,\delta)$ is not good, then
    there does not exist any invariant game having
    $\mathcal{P}$ as set of $P$-positions.
\end{theorem}

\begin{proof}
    Assume that $\gamma<0$ and there exists a prefix $p=w_0\cdots
    w_{n-1}\in \{a,b\}^*$ of $\mathbf{w}_-$ of length $n$ less than
    $B_1$ such that $p$ is suffix-bad. Therefore, since
    $\pi_1(a)=\pi_1(b)=1$, the first $n+1$  elements in $\mathcal{P}$
    are $(i,B_i)$ for $i=0,\ldots,n$.

Assume first that
    all the suffixes of $p$ satisfy (B.1). Consider the position
    $(n,B_n-1)$ which is not in $\mathcal{P}$. We will show that from
    $(n,B_n-1)$, there is no allowed move leading to a position in
    $\mathcal{P}$. Proceed by contradiction. Assume that there exists
    $i<n$ such that we can play the move $(n,B_n-1)\to (A_i,B_i)$. It
    implies that $(n-A_i, B_n-B_i-1)$ is not in
    $\mathcal{P}-\mathcal{P}$. By assumption, the suffix $w_i\cdots
    w_{n-1}$ satisfies (B.1). Hence there exists a factor $f=w_{j+i}\cdots
    w_{j+n-1}\in\{a,b\}^*$ in $w$ of length $n-i=n-A_i$ such that
    $|f|_b<|w_i\cdots w_{n-1}|_b$. Since $\pi_1(a)=\pi_1(b)=1$, we get $A_{j+n}-A_{j+i}=\sum_{k=i}^{n-1} \pi_1(w_{j+k})=n-i$.
Note that $\pi_2(f)$ and
    $\pi_2(w_i\cdots w_{n-1})$ are factors of the Sturmian word
    $\mathbf{s}_{\beta,\delta}$. Hence, the balance property implies that
    $|f|_b-1=|w_i\cdots w_{n-1}|_b$. From this, it follows that
$$B_{j+n}-B_{j+i}=\sum_{k=i}^{n-1} \pi_2(w_{j+k})=\sum_{k=i}^{n-1} \pi_2(w_k)-1=B_n-B_i-1$$
and $(A_{j+n}-A_{j+i},B_{j+n}-B_{j+i})=(n-A_i, B_n-B_i-1)$ is in $\mathcal{P}-\mathcal{P}$.

Assume now that all the suffixes of $p$ satisfy (B.2). One can proceed
in a similar way to prove that from the position $(n,B_{n-1}-1)$ which
is not in $\mathcal{P}$, there is no allowed move leading to a
position in $\mathcal{P}$.

Assume that $\gamma>0$ and there exists a prefix $p=w_0w_1\cdots
w_{n-1}$ of $\mathbf{w}_+$ such that $2\le |p|<B_1$, $w_1\cdots
w_{n-1}\in\{a,b\}^*$ and $p$ is bad. Therefore, the first $n+1$
elements in $\mathcal{P}$ are $(i,B_i)$ for $i=0,\ldots,n$. Since $p$
is bad, from Definition~\ref{def:pbad}, $B_n-B_{n-1}=\pi_2(w_{n-1})\in
2B_1+\{-1,0,1\}$ and $B_1\ge 3$. Hence $B_{n-1}+B_1<B_n$ and
$(n,B_{n-1}+B_1)$ is not in $\mathcal{P}$.  We will show that from
this position $(n,B_{n-1}+B_1)$, there is no allowed move leading to a
position in $\mathcal{P}$. We proceed by contradiction. Assume first
that there is a move $(n,B_{n-1}+B_1)\to(0,0)$. Since $p$ is bad,
using Remark~\ref{rem:bad+}, there exists a factor $f=f_1\cdots
f_{n-1}$ occurring in $\mathbf{w}_+$ such that $\sum_{k=1}^{n-1}
f_k=(n,B_{n-1}+B_1)$. In other words, $(n,B_{n-1}+B_1)$ belongs to
$\mathcal{P}-\mathcal{P}$. Now assume that there is a move
$(n,B_{n-1}+B_1)\to(A_j,B_j)$, with $j>0$.  Since $p$ is bad,
$w_j\cdots w_{n-2}$ is the reversal of $w_1\cdots w_{n-j-1}$. Hence we
get
$\sum_{k=j}^{n-2}\pi_2(w_k)+B_1=\sum_{k=1}^{n-j-1}\pi_2(w_k)+\pi_2(w_0)=B_{n-1}+B_1-B_j$
and $(n-j,B_{n-1}+B_1-B_j)$ belongs to $\mathcal{P}-\mathcal{P}$.
\end{proof}

We now turn to the converse. We show that, if the $4$-tuple
$(\alpha,\beta,\gamma,\delta)$ is good, then an invariant game having
$\mathcal{P}$ as set of $P$-positions exists. Moreover, the set of
allowed moves can be chosen to be maximal and taken as
$$\mathcal{M}:=\mathbb{N}^2\setminus(\mathcal{P}-\mathcal{P}).$$
Indeed, in an invariant game, any move can be played from any position
(with the only restriction that there is enough tokens left).
Therefore, a move between two $P$-positions is never allowed.

\begin{lemma}\label{lem4} If the $4$-tuple $(\alpha,\beta,\gamma,\delta)$ is good, then we have
   $$ \{(k,0),(0,k)\mid k\ge 1\}\cup\{(1,k)\mid 1\le k < B_1\}\subseteq\mathcal{M}.$$
\end{lemma}

\begin{proof}
    Since $\{A_n\mid n\ge 1\}$ and $\{B_n\mid n\ge 1\}$ make a
    partition of $\mathbb{N}$, we clearly have that $(k,0)$ and
    $(0,k)$ belong to $\mathcal{M}$.

    Let $k$ be such that $1\le k < B_1$. By way of contradiction,
    assume that $(1,k)\in \mathcal{P}-\mathcal{P}$.  Note that
    $\delta<1$ (see Remark~\ref{rem2}), implies $B_1=\lfloor\beta+\delta\rfloor \le \lfloor\beta\rfloor+1$.  We consider
    two cases:
    \begin{itemize}
      \item $(1,k)=(A_n-A_m,B_n-B_m)$ for some $n>m\geq 0$.  According
        to Lemma~\ref{lem:ecarts}, $A_n-A_m=1$ implies $n=m+1$ and
        $(A_n-A_m,B_n-B_m)$ can possibly take the values $(1,B_1)$,
        $(1,\lfloor\beta\rfloor)$ and $(1,\lfloor\beta\rfloor+1)$. Since $k<B_1$, if
        $k<\lfloor\beta\rfloor$, then $(1,k)$ cannot be of this form.
        If $k=\lfloor\beta\rfloor$, then $k < B_1$ implies that
        $B_1=\lfloor\beta\rfloor+1$ and thus $\gamma<0$. In other words,
        $b$ is a prefix of $\mathbf{w}_-$. By assumption,
        $(\alpha,\beta,\gamma,\delta)$ is good (and we are in the
        situation where $\gamma<0$), this means that $\mathbf{w}_-$
        contains no occurrence of $a$ (coding the difference
        $(1,\lfloor\beta\rfloor)$) because otherwise, the prefix $b$
        of $\mathbf{w}_-$ would satisfy (B.1).
      \item $(1,k)=(A_n-B_m,B_n-A_m)$ for some $n>m\geq 0$. Hence we
        have $B_n-A_m=(B_n-B_m)+(B_m-A_m)\geq B_1$, contradicting the
        hypothesis.
    \end{itemize}
\end{proof}

\begin{lemma}\label{lem6}
    Let $n,m,k,l,\Delta$ be five integers such that $n>m\geq 0$,
    $k>l\geq 0$, and $\Delta\geq 2$. If $A_n-A_m=A_k-A_l+\Delta$, then
    we have $B_n-B_m>B_k-B_l+\Delta-4$.
\end{lemma}
\begin{proof}
    Consider the relative position of the straight line of equation
    $y=\frac{\beta}{\alpha}x+2\delta$ and the point of coordinates
    $(A_n,B_n)$. More precisely, consider the difference of
    $y$-coordinates between the point of the line having a
    $x$-coordinate equal to $A_n$ and the point $(A_n,B_n)$:
$$\frac{\beta}{\alpha}A_n+2\delta-B_n=\frac{\beta}{\alpha}\lfloor n\alpha+\gamma\rfloor+2\delta-\lfloor n\beta+\delta\rfloor.$$
This quantity is equal to
$$\frac{\beta}{\alpha}(n\alpha+\gamma)+2\delta-(n\beta+\delta)-\frac{\beta}{\alpha}\{n\alpha+\gamma\}+\{n\beta+\delta\}=-
\frac{\beta}{\alpha}\{n\alpha+\gamma\}+\{n\beta+\delta\}$$ where, for
the last equality, we have used \eqref{eq:r4}. Hence, this difference
satisfies
$$\frac{\beta}{\alpha}A_n+2\delta-B_n\in\left(-\frac{\beta}{\alpha},1\right), \ i.e.,\  \frac{\beta}{\alpha}A_n+2\delta-1 <B_n<\frac{\beta}{\alpha}A_n+2\delta +\frac{\beta}{\alpha}.$$

Assume that $A_k-A_l=i$. From the above computation, we get
$$B_k-B_l<\frac{\beta}{\alpha}A_k+2\delta+\frac{\beta}{\alpha} -(\frac{\beta}{\alpha}A_l+2\delta-1)=(i+1) \frac{\beta}{\alpha}+1.$$
Similarly, by assumption $A_n-A_m=i+\Delta$ and we get
$$B_n-B_m>\frac{\beta}{\alpha}A_n+2\delta-1-(\frac{\beta}{\alpha}A_m+2\delta+\frac{\beta}{\alpha})=(i+\Delta-1)\frac{\beta}{\alpha}-1.$$
To get the conclusion, we have to show that
$$(i+\Delta-1)\frac{\beta}{\alpha}-1\ge (i+1) \frac{\beta}{\alpha}+1-\Delta$$
which is equivalent to $\Delta-2\ge (2-\Delta)\alpha/\beta$. Observe that this last inequality trivially holds true for $\Delta\ge 2$.
\end{proof}

\begin{remark}\label{rem:ajout}
    The map $n\mapsto B_n-A_n$ is non-decreasing. Indeed, since $A_{n+1}-A_n\in\{1,2\}$, $B_{n+1}-B_n\in\{\lfloor\beta\rfloor,\lfloor\beta\rfloor+1\}$, then $B_{n+1}-A_{n+1}\ge B_{n}-A_{n}+\lfloor\beta\rfloor-2$ and the conclusion follows from the fact that $\lfloor\beta\rfloor\ge 2$.
\end{remark}

\begin{lemma}\label{lem:ajout}
	If the $4$-tuple $(\alpha,\beta,\gamma,\delta)$ is good and $\lfloor \beta \rfloor =2$, then the following properties hold:
	\begin{enumerate}
	\item $B_1=3$.
	\item $\gamma <0 $.
	\item $\mathbf{w}_{-}$ starts with $bb$. The factor $a$ never appears in $\mathbf{w}_{-}$. 
	\item The first three pairs of the sequence $(A_n,B_n)$ are $(0,0)$, $(1,3)$, $(2,6)$.
    \item The factor $cc$ never occurs in $\mathbf{w}_{-}$.
	\end{enumerate}
\end{lemma}

\begin{proof}
\begin{enumerate}
\item Obtained from $B_1\geq 3$ and $\delta<1$.
\item Since $B_1>\lfloor \beta \rfloor$, we have $\delta>0$, and thus $\gamma<0$.
\item Since $(A_1,B_1)=(1,3)$, it means that $w_0=b$ and we necessarily have $A_2=2$. Since $(\alpha,\beta,\gamma,\delta)$ is good, the prefix $b$ is not suffix-bad, meaning that $a$ never appears in $\mathbf{w}_{-}$. In other words, $w_1=b$.
\item Directly deduced from the previous item.
\item Since $bb$ occurs in $\mathbf{w}_{-}$ and $\pi_2(\mathbf{w}_{-})$ is Sturmian, it means that $cc$ never appears.
\end{enumerate}
\end{proof}

\begin{theorem}\label{the:principal}
    If the $4$-tuple $(\alpha,\beta,\gamma,\delta)$ is good, then the
    invariant game having $\mathcal{M}$ as set of moves
    admits $\mathcal{P}$ as set of $P$-positions.
\end{theorem}

\begin{proof}
    By construction of $\mathcal{M}$, it is clear that from a position
    in $\mathcal{P}$ any move leads to a position in
    $\mathcal{N}=\mathbb{N}^2\setminus\mathcal{P}$.  Now we show that
    if $(x,y)$ belongs to $\mathcal{N}$, there exists a move
    $m\in\mathcal{M}$ such that $(x,y)-m$ belongs to $\mathcal{P}$. If
    $x=0$ or $y=0$, we conclude directly using Lemma~\ref{lem4}.
    Without loss of generality, we now may assume that $0<x\le y$.

    Since $\{A_n\mid n\ge 1\}$ and $\{B_n\mid n\ge 1\}$ make a
    partition of $\mathbb{N}$, we consider three cases.

\smallskip
\noindent {\bf Case 1)} If $x=B_i$ for some $i> 0$, then consider the
move $(0,y-A_i)$.  From Lemma~\ref{lem4}, this move belongs to
$\mathcal{M}$ since $y\ge x= B_i >A_i$.  Hence the resulting position
$(x,y)-(0,y-A_i)=(B_i,A_i)$ belongs to $\mathcal{P}$.

\smallskip
\noindent {\bf Case 2)} If $x=A_i$ and $y>B_i$, then consider the move
$(0,y-B_i)\in\mathcal{M}$ leading to the position $(A_i,B_i)$.

\smallskip
\noindent {\bf Case 3)} Consider the case where $x=A_i\le y<B_i$. Note
that we do not take into account the case $y=B_i$, since we would have
$(x,y)=(A_i,B_i)$ which does not belong to $\mathcal{N}$.  We consider
two sub-cases:

\smallskip
\noindent {\bf Case 3.1)} $B_1<x=A_i\le y<B_i$. 

\smallskip
\noindent
{\bf 3.1.a)} If $y-x \geq 2$, we aim to show that it
is always possible to move either to $(1,B_1)$ or to $(B_1,1)$. Since the move should not belong to $\mathcal{P}-\mathcal{P}$, we have three situations that may occur.

{\em Situation 1}. We
start by proving that these two moves are not of the form
$(A_n-B_m,B_n-A_m)$ with $n>m>0$. First assume that $(A_i-1,y-B_1)=(A_n-B_m,B_n-A_m)$ for some $n>m>0$.
Hence we have $$A_i-1=A_n-B_m\quad \text{ and }\quad y-B_1=B_n-A_m.$$
In that case, since $A_n-A_i=B_m-1\geq 0$, note that we necessarily
have $i\leq n$.  We now subtract the previous two equalities to obtain
\[
y=B_n-A_n+B_m-A_m+A_i-1+B_1
\]
And since $y<B_i$, the following inequality holds
\[
B_n-A_n+B_m-A_m<B_i-A_i+A_1-B_1
\]
Since $n\geq i$, the contradiction is guaranteed thanks to Remark~\ref{rem:ajout} and $A_1-B_1<0$.
		
Now suppose that $(A_i-B_1,y-1)=(A_n-B_m,B_n-A_m)$ for some $n>m>0$.
With the same argument as above, we get
\[
B_n-A_n+B_m-A_m<B_i-A_i+B_1-A_1
\]
Since we still have $n\geq i$ in that case ($A_n-A_i$ being positive),
and $m\geq 1$, the contradiction is ensured thanks to Remark~\ref{rem:ajout}.

{\em Situation 2}. Now assume that $(A_i-1,y-B_1)=(B_n-A_m,A_n-B_m)$ for some $n>m>0$. Again, subtracting the two equalities, we get
$$y-A_i=B_1-1-(B_m-A_m)-(B_n-A_n).$$
The left hand side is non-negative, but thanks to Remark~\ref{rem:ajout}, the right hand side is negative.
Similarly, if $(A_i-B_1,y-1)=(B_n-A_m,A_n-B_m)$ for some $n>m>0$, we get directly $y-A_i=1-B_1-(B_m-A_m)-(B_n-A_n)$. This is a contradiction because the l.h.s. is non-negative and the r.h.s. is negative.

{\em Situation 3}. It now remains to prove that at least one of the two moves is neither of
the form $(A_n-A_m,B_n-B_m)$, nor $(B_n-B_m,A_n-A_m)$ with $n>m\geq 0$.

Case 1: $A_i-1\leq y-B_1$. From Lemma \ref{lem:ecarts}, $B_n-B_m\ge (n-m)\lfloor\beta\rfloor$ and $A_n-A_m\le 2(n-m)$. Hence 
$(A_i-1,y-B_1)$ cannot be of the form $(B_n-B_m,A_n-A_m)$, except in the particular case where $A_i-1=y-B_1$ and $B_n-B_m=A_n-A_m$. Again from Lemma \ref{lem:ecarts}, this may only happen if $\lfloor\beta\rfloor=2$, together with $A_n-A_m=2(n-m)$ and $B_n-B_m= (n-m)\lfloor\beta\rfloor$. But from Lemma \ref{lem:ajout}, the factor $cc$ never occurs in $\mathbf{w}_{-}$ in that case, implying $n-m=1$, then $A_i=3$, contradicting $A_i>B_1$.

Hence $(A_i-1,y-B_1)$ can only be of the form $(A_n-A_m,B_n-B_m)$. If it is the case, we will show that the move $(A_i-B_1,y-1)$ is not forbidden. By way of contradiction, assume that both moves are forbidden, i.e., there exists $n>m\geq 0$ and $k>l\geq 0$
satisfying
$$x-1=A_n-A_m\quad \text{ and }\quad y-B_1=B_n-B_m,$$
$$x-B_1=A_k-A_l\quad \text{ and }\quad y-1=B_k-B_l.$$

Indeed, since $A_i-B_1<y-1$, then $(A_i-B_1,y-1)$ cannot be of the form $(B_k-B_l,A_k-A_l)$ because $B_k-B_l\ge A_k-A_l$. Moreover, the other forms $(A_k-B_l,B_k-A_l)$ or $(B_k-A_l,A_k-B_l)$ are also excluded thanks to the above discussion in situations 1 and 2.
Thus, if $(A_i-B_1,y-1)$ is forbidden, it can only be of the form $(A_k-A_l,B_k-B_l)$ .

Now, by subtracting the last two equalities from the first two ones, we
obtain
$$A_n-A_m=A_k-A_l+B_1-1\quad \text{ and }\quad
B_n-B_m=B_k-B_l-(B_1-1).$$ Since $B_1-1\geq 2$, Lemma~\ref{lem6} with
$\Delta=B_1-1$ yields a contradiction.

Case 2: $A_i-1> y-B_1$. Hence $(A_i-1,y-B_1)$ can only be of the form $(B_n-B_m,A_n-A_m)$. Similarly, $(A_i-B_1,y-1)$ satisfies $A_i-B_1<y-1$, and thus can only be of the form $(A_n-A_m,B_n-B_m)$. Now assume that both moves are forbidden, i.e., there exists $n>m\geq 0$ and $k>l\geq 0$
satisfying
$$x-1=B_n-B_m\quad \text{ and }\quad y-B_1=A_n-A_m,$$
$$y-1=B_k-B_l \quad \text{ and }\quad x-B_1=A_k-A_l.$$

By subtracting the last two equalities from the first two ones, we
obtain
$$A_n-A_m=A_k-A_l+y-x\quad \text{ and }\quad
B_n-B_m=B_k-B_l-(y-x).$$ Since $y-x\geq 2$, Lemma~\ref{lem6} with
$\Delta=y-x$ yields a contradiction.

\smallskip
\noindent
{\bf 3.1.b)} If $y-x<2$, we will show that playing from $(x,y)=(A_i,y)$ to $(0,0)$ is ``almost always'' legal. Indeed, if this move was forbidden, there would exist a factor $f_1\cdots f_l$ in $\mathbf{w}$, of length $l>1$ (since $A_i>B_1\geq 3$), and satisfying $y-x=\sum_{k=1}^{l}\pi_2(f_k)-\sum_{k=1}^{l}\pi_1(f_k)<2$. According to Lemma \ref{lem:ecarts}, the only factors that may satisfy these conditions need to have $\lfloor \beta \rfloor =2$ and are:
\begin{itemize}
\item factors of the kind $cc^+$: According to Lemma \ref{lem:ajout}, $cc$ never occur in $\mathbf{w}_{-}$, which excludes such factors.
\item factors of the kind $ac^+$ and their permutations: According to Lemma \ref{lem:ajout}, these factors never occur because they contain $a$.
\item factors of the kind $dc^+$ and their permutations: According to Lemma \ref{lem:ajout}, since $cc$ never occur, the list reduces to the factors $\{dc,cd,cdc\}$. The case $cdc$ would lead to study $(A_i,A_i+1)=(6,7)$, which is impossible since $B_2=6$ from Lemma \ref{lem:ecarts}. The cases $cd$ and $dc$ lead to examinate position $(A_i,A_i+1)=(4,5)$. For this particular case, play $(4,5)\rightarrow (3,1)$, which is legal since $(1,4)=(1,\lfloor \beta \rfloor+2)$ is clearly not in $\mathcal{P}-\mathcal{P}$ from Lemma \ref{lem:ecarts}.
\end{itemize}

\smallskip
\noindent {\bf Case 3.2)} $x=A_i\le y<B_i$ and $A_i<B_1$. If $i=1$,
then $(x,y)=(1,y)$ and thanks to Lemma~\ref{lem4}, playing to $(0,0)$
is allowed.  Hence in the following discussion, we may assume that
$i>1$.  Since $A_i<B_1$ and the increasing sequences $(A_n)_{n\ge 1}$
and $(B_n)_{n\geq 1}$ make a partition of $\mathbb{N}_{>0}$, we have
$A_{i'}=i'$ for $0<i'\leq i$.

If $y=B_j$ for some $j<i$, playing to $(A_j,B_j)$ is allowed from Lemma~\ref{lem4}.

\smallskip
\noindent
{\bf 3.2.a)}
If $y< B_{i-2}$, there exists $t\ge 3$ such that $B_{i-t}<y<B_{i-t+1}$
and $i\ge 3$. Playing to $(A_{i-t},B_{i-t})$ is allowed. Indeed, we have
$$(x,y)-(A_{i-t},B_{i-t})=(t,k) \text{ where }t\ge 3, k\le\lfloor\beta\rfloor$$
and the conclusion follows from Lemma~\ref{lem:ecarts}: $(t,k)$  with $t\ge 3, k\le\lfloor\beta\rfloor$ is not of the form $(A_n-A_m,B_n-B_m)$. In addition, it is easy to see that $(t,k)$ can neither be of the form $(A_n-B_m,B_n-A_m)$ (otherwise we would have $k=B_n-A_m\geq \lfloor\beta\rfloor +1$), nor of the form $(B_n-A_m,A_n-B_m)$ (otherwise we would have $t=B_n-A_m\geq B_1$, contradicting the property $t=A_i-A_{i-t}<B_1$), nor $(B_n-B_m,A_n-A_m)$ (otherwise we would have $B_n-B_m=t<i=A_i<B_1$, implying $B_n-B_m < \lfloor\beta\rfloor$ !)

\smallskip
\noindent
{\bf 3.2.b)} If $B_{i-2}<y<B_{i-1}-1$ or if $y=B_{i-1}-1$ and $B_{i-1}-B_{i-2}=\lfloor\beta\rfloor$, then playing to $(A_{i-2},B_{i-2})$ is allowed. Indeed, we have
$$(x,y)-(A_{i-2},B_{i-2})=(2,k) \text{ where }k<\lfloor\beta\rfloor$$
and, as in the previous case, the conclusion follows again from Lemma~\ref{lem:ecarts}.

\smallskip
\noindent {\bf 3.2.c)} If $y=B_{i-1}-1$ and
$B_{i-1}-B_{i-2}=\lfloor\beta\rfloor+1$, we consider two cases
according to the sign of $\gamma$.
\begin{itemize}
  \item[$\gamma>0$:] We will show that moving to $(0,0)$ is always
    possible. For this purpose, it suffices to show that there exists
    no factor $f=f_1\cdots f_l$ of $\mathbf{w}_+$ such that
    $\sum_{k=1}^{l}f_k=(i,B_{i-1}-1)$. Assume that such a factor $f$
    exists. In other words, $(i,B_{i-1}-1)$ is of the kind\footnote{In
      all what follows, the case $(A_n-B_m,B_n-A_m)$ is left to the
      reader. Note that we do not have to consider moves $(B_n-B_m,A_n-A_m)$ nor $(B_n-A_m,A_n-B_m)$ because $B_n-B_m\ge A_n-A_m$ and $B_n-A_m\ge A_n-B_m$ but we are in a position $(i,y)$ with $y<B_{i}$. To avoid lengthy discussions, we have only considered the more intricate situation.}
$(A_n-A_{n-l},B_n-B_{n-l})$ and belongs to
    $\mathcal{P}-\mathcal{P}$.  Necessarily the length of $f$
    satisfies $l\leq i$ (because $\pi_1(f)\in\{1,2\}^*$).

    If $l=i$, since $s_{\beta,\delta}(0)w_1w_2\cdots$ is a Sturmian
    word, one can use Remark~\ref{rem:heavy-light} and the fact that
    $s_{\beta,\delta}(0)>\pi_2(w_0)$, to get
    $$B_{i-1}-1=\sum_{k=1}^{l}\pi_2(f_k)\geq \sum_{k=0}^{i-1}\pi_2(w_k).$$  As
    \eqref{eq:sum} gives
    $\sum_{k=0}^{i-1}\pi_2(w_k)=B_{i-1}+\pi_2(w_{i-1})$, we get a
    contradiction.

    If $l=i-1$, we conclude in the same way that
    $\sum_{k=1}^{l}\pi_2(f_k)\geq \sum_{k=0}^{i-2}\pi_2(w_k)=B_{i-1}$,
    a contradiction.

    If $l<i-1$, then using Remark~\ref{rem:heavy-light}, we get
$$\sum_{k=1}^{l}\pi_2(f_k)\le \sum_{k=0}^{i-3}\pi_2(s_{\beta,\delta}(k))+1=\pi_2(w_0)-\lfloor\delta\rfloor+\sum_{k=1}^{i-3}\pi_2(s_{\beta,\delta}(k))+1\le B_{i-2}+2$$
where for the last inequality, we have used the fact that
$\lfloor\delta\rfloor\le -1$. Recall that we are assuming here that
$B_{i-1}-B_{i-2}=\lfloor\beta\rfloor+1$ (but we also have
$\lfloor\beta\rfloor\geq 3$ since $\delta<0$ and $B_1\geq 3$). Hence,
$\sum_{k=1}^{l}\pi_2(f_k)\le B_{i-2}+2<B_{i-1}-1$ which is again a
contradiction.


  \item[$\gamma<0$:] Consider the prefix $p$ of $\mathbf{w}_-$ of
    length $i$. Since $A_{j}-A_{j-1}=1$ for all $0<j\leq i$, we know
    that $p\in\{a,b\}^*$ and $|p|<B_1$. Since
    $(\alpha,\beta,\gamma,\delta)$ is good, the prefix $p$ is not
    suffix-bad. In particular, it means that there exists some
    $j\in\{0,\ldots,i-2\}$ such that $w_j\cdots w_{i-1}$ does not
    satisfy property (B.2). Playing from $(x,y)$ to $(A_{j},B_{j})$
    is thus allowed. Indeed, assume on the contrary that
    $(x-A_{j},y-B_{j})$ belongs to $\mathcal{P-P}$ and is of the form
    $(A_n-A_m,B_n-B_m)$ (the other case $(A_m-B_n,B_m-A_n)$ cannot
    occur). It would mean that there exists a factor $f=f_1\cdots f_l$
    of $\mathbf{w}_-$ such that
    $\sum_{k=1}^{l}f_k=(i-j,B_{i-1}-1-B_j)$. Necessarily we have
    $l\leq i-j$.  If $l=i-j$, then $\sum_{k=1}^{l}\pi_2(f_k)\geq
    \sum_{k=j}^{i-1}\pi_2(w_k)-1$, since $\pi_2(\mathbf{w}_-)$ is
    Sturmian. But we also have
    $\sum_{k=1}^{l}\pi_2(f_k)=B_{i-1}-1-B_j=
    \sum_{k=j}^{i-1}\pi_2(w_k)-1$, leading to a contradiction.  If
    $l<i-j-1$, then $\sum_{k=1}^{l}\pi_2(f_k)\leq
    \sum_{k=j}^{i-3}\pi_2(w_k)+1$. But we also have
    $\sum_{k=1}^{l}\pi_2(f_k)=B_{i-1}-1-B_j=
    \sum_{k=j}^{i-1}\pi_2(w_k)-1$, leading to a contradiction.  Hence
    $l=i-j-1$, meaning that $f$ is a factor of $\mathbf{w}_-$ over
    $\{a,b,c,d\}^*$ of length $(i-j-1)$ with exactly $|f|_c+|f|_d=1$.
    In particular, $\sum_{k=1}^{l}\pi_1(f_k)=(i-j)$.  Assume first
    that $|f|_c=1$. Since $w_j\cdots w_{i-1}$ does not satisfy
    property (B.2), we get $|f|_a= |w_j\cdots w_{i-2}|_a-1$, $|f|_b=
    |w_j\cdots w_{i-2}|_b$, and
    $\sum_{k=1}^{l}\pi_2(f_k)=\sum_{k=j}^{i-2}\pi_2(w_k)$, which is a
    contradiction. The proof is similar in the case $|f|_d=1$.
\end{itemize}

\smallskip
\noindent
{\bf 3.2.d)}
If $B_{i-1}<y<B_{i}-1$ or if $y=B_{i}-1$ and $B_i-B_{i-1}=\lfloor\beta\rfloor$, then consider the move
$$(x,y)-(A_{i-1},B_{i-1})=(1,k) \text{ where }k<\lfloor\beta\rfloor.$$
\begin{itemize}
\item[$\gamma<0$:] This move is allowed since $(1,k)$ never occurs in $\mathbf{w}_-$.
\item[$\gamma>0$:] This move is allowed whenever $k\neq B_1$
  (recall that $(1,B_1)$ occurs as $w_0$ in $\mathbf{w}_+$). Now
  assume that $y=B_{i-1}+B_1$. In particular, we get $B_1\le \lfloor\beta\rfloor -1$.

  In the case where $w_1\cdots w_{i-2}$ is a palindrome, we show that
  we can play to $(0,0)$.  Assume to the contrary that there exists a
  factor $f=f_1\cdots f_l$ occurring in $\mathbf{w}_+$ such that
  $\sum_{k=1}^lf_i=(x,y)$.  Necessarily, $l\le i$. If $l=i$, then
  $\sum_{k=1}^l \pi_2(f_k)\ge \sum_{k=1}^{i-2}
  \pi_2(w_k)+2\lfloor\beta\rfloor$ since $\mathbf{s}_{\beta,\delta}$
  is Sturmian. But we also have $\sum_{k=1}^l
  \pi_2(f_k)=y=\sum_{k=1}^{i-2} \pi_2(w_k)+2B_1$ contradicting the
  fact that $B_1<\lfloor\beta\rfloor$. If $l<i-1$, then $\sum_{k=1}^l
  \pi_2(f_k)\le \sum_{k=1}^{i-2} \pi_2(w_k)+1$. But we also have
  $\sum_{k=1}^l \pi_2(f_k)=y=\sum_{k=1}^{i-2} \pi_2(w_k)+2B_1$ leading
  to the contradiction $2B_1\le 1$.

  Hence $l=i-1$ and using again the Sturmian property, we get
  $\sum_{k=1}^{i-1} \pi_2(w_k)-1 \le \sum_{k=1}^l \pi_2(f_k)\le
  \sum_{k=1}^{i-1} \pi_2(w_k)+1$. But $\sum_{k=1}^l
  \pi_2(f_k)=y=\sum_{k=1}^{i-1} \pi_2(w_k)+2B_1-B_i+B_{i-1}$ leading
  to $$B_i-B_{i-1}-1\le 2B_1\le B_i-B_{i-1}+1.$$  The prefix $w_0\cdots
  w_{i-1}$ is not bad.  According to Definition~\ref{def:pbad}, the
  situation can be split into three cases. Assume first that
  $2B_1=\pi_2(w_{i-1})-1$. The factor $f$ of length $i-1$ is such that
  $\sum_{k=1}^l\pi_1(f_k)=i$, this implies that $|f|_c+|f|_d=1$.
  Moreover,
  $\sum_{k=1}^l\pi_2(f_k)=y=\sum_{k=1}^{i-1}\pi_2(w_k)+2B_1-\pi_2(w_{i-1})$.
  Hence $\sum_{k=1}^l\pi_2(f_k)=\sum_{k=1}^{i-1}\pi_2(w_k)-1$ and the
  factor $f$ must satisfy $|w_1\cdots w_{i-2}|_a=|f|_a$, $|w_1\cdots
  w_{i-2}|_b-1=|f|_b$, $|f|_c=1$. This contradicts the fact that
  $w_0\cdots w_{i-1}$ is not bad. The last two cases where
  $2B_1=\pi_2(w_{i-1})$ and $2B_1=\pi_2(w_{i-1})+1$ are treated
  similarly.

  If $w_1\cdots w_{i-2}$ is not a palindrome, consider the smallest
  $j$ such that $w_{1+j}\neq w_{i-2-j}$. We can play to
  $(A_{i-j-2},B_{i-j-2})$ except if there exists some factor
  $f=w_{t}\cdots w_{t+l-1}$ occurring in $\mathbf{w}_+$ and satisfying $$\sum_{k=t}^{t+l-1}
  w_k=\sum_{k=i-j-2}^{i-2} w_k+(1,B_1).$$ Assume to conclude this part
  of the proof that we are in this latter situation.

  Clearly we have that $l\leq j+2$. If $t=0$, since
  $\pi_1(w_0)=\cdots=\pi_1(w_{l-1})=1$, then $l=j+2$, and
  $\sum_{k=t}^{t+l-1} \pi_2(w_k) \neq \sum_{k=i-j-2}^{i-2}
  \pi_2(w_k)+B_1$. If $t\neq 0$ and $l<j+2$, from
  Remark~\ref{rem:heavy-light}, we have $\sum_{k=t}^{t+l-1}
  \pi_2(w_k)-1\le \sum_{k=i-j-2}^{i-2} \pi_2(w_k)$. Hence,
  $\sum_{k=t}^{t+l-1} \pi_2(w_k)< \sum_{k=i-j-2}^{i-2}
  \pi_2(w_k)+B_1$. If $t\neq 0$ and $l=j+2$, from
  Remark~\ref{rem:heavy-light}, we have $\sum_{k=t}^{t+l-1}
  \pi_2(w_k)\geq \sum_{k=i-j-2}^{i-1} \pi_2(w_k)-1$ and as
  $\sum_{k=t}^{t+l-1} \pi_2(w_k)=\sum_{k=i-j-2}^{i-1}
  \pi_2(w_k)+B_1-\pi_2(w_{i-1})$. Hence we obtain $B_1\geq \pi_2(w_{i-1})-1$,
  which is a contradiction if $B_1\neq \lfloor\beta\rfloor-1$ or
  if $\pi_2(w_{i-1})=\lfloor\beta\rfloor+1$. In the case where
  $B_1=\lfloor\beta\rfloor-1$ and
  $\pi_2(w_{i-1})=\lfloor\beta\rfloor$, since $B_1\geq 3$, we get
  $\lfloor\beta\rfloor\geq 4$.  In this particular case, none of the
  three situations described in Definition~\ref{def:pbad} can occur.
  Indeed, we have
  $2B_1=2\lfloor\beta\rfloor-2>\lfloor\beta\rfloor+1=\pi_2(w_{i-1})+1$.
  Hence one can use the strategy described above (i.e., playing to
  $(0,0)$), by acting as if $w_1\cdots w_{i-2}$ was a palindrome.

  \end{itemize}

\smallskip
\noindent
{\bf 3.2.e)}
If $y=B_{i}-1$ and $\pi_2(w_{i-1})=B_i-B_{i-1}=\lfloor\beta\rfloor+1$, we consider two cases.

\begin{itemize}
  \item[$\gamma>0$:] We will show that moving to $(0,0)$ is always
    possible. For this purpose, it suffices to show that there exists
    no factor $f=f_1\cdots f_l$ of $\mathbf{w}_+$ such that
    $\sum_{k=1}^{l}f_k=(i,B_{i}-1)$. Assume that such a factor $f$
    exists. Necessarily $f$ satisfies $l\leq i$. If $l<i$, then
    $\sum_{k=1}^{l}\pi_2(f_k)\leq
    \sum_{k=0}^{i-2}\pi_2(w_k)+2=B_{i-1}+2$, yielding a contradiction.
    If $l=i$, we have $\sum_{k=1}^{l}\pi_2(f_k)\geq
    \sum_{k=1}^{i-1}\pi_2(w_k)-1+\lfloor\beta\rfloor$ because
    $w_1w_2\cdots$ is Sturmian. Since we must have
    $y=B_1+\sum_{k=1}^{i-1}\pi_2(w_k)-1=\sum_{k=1}^{l}\pi_2(f_k)$, it
    implies $B_1\geq \lfloor\beta\rfloor$. As
    $B_1<\lfloor\beta\rfloor+1$ when $\gamma>0$, we have
    $B_1=\lfloor\beta\rfloor$. We deduce that
    $s_{\beta,\delta}(0)=\lfloor\beta\rfloor+1$, and
    $\sum_{k=0}^{i-1}s_{\beta,\delta}(k)=\sum_{k=0}^{i-1}\pi_2(w_k)+1=B_i+1$.
    Since a Sturmian word is balanced, no factor of length $i$ gives a
    sum equal to $B_i-1$.

\item[$\gamma<0$:] We consider the prefix $p$ of $\mathbf{w}$ of length $i$. Since
$A_{j}-A_{j-1}=1$ for all $0<j\leq i$, we know that $p\in\{a,b\}^*$ and
$|p|<B_1$. Since $(\alpha,\beta,\gamma,\delta)$ is good, the prefix
$p$ is not suffix-bad. In particular, it means that there exists some
$j\in\{0,\ldots,i-1\}$ such that $w_j\cdots w_{i-1}$ does not satisfy
property (B.1). Playing from $(x,y)$ to $(A_{j},B_{j})$ is thus
allowed. Indeed, assume on the contrary that $(x-A_{j},y-B_{j})$
belongs to $\mathcal{P-P}$ and is of the form $(A_n-A_m,B_n-B_m)$ (the
other case $(A_m-B_n,B_m-A_n)$ cannot occur). It would mean that there
exists a factor $f=f_1\cdots f_l$ of $\mathbf{w}$ such that
$\sum_{k=1}^{l}f_k=(i-j,B_i-1-B_j)$. Necessarily we have $l\leq i-j$.
If $l<i-j$, then $\sum_{k=1}^{l}\pi_2(f_k)\leq
\sum_{k=j}^{i-2}\pi_2(w_k)+1$ since $\pi_2(\mathbf{w})$ is Sturmian.
But $\sum_{k=1}^{l}\pi_2(f_k)=B_i-1-B_j=\sum_{k=j}^{i-1}\pi_2(w_k)-1$,
leading to the contradiction $\pi_2(w_{i-1})\leq 2$. Hence $l=i-j$, meaning that $f$ is a factor of $\mathbf{w}$ over $\{a,b\}^*$ of length $(i-j)$. Since $w_j\cdots w_{i-1}$ does not satisfy
property (B.1), we get $|f|_b\geq |w_j\cdots w_{i-1}|_b$. In other words, we have $\sum_{k=1}^{l}\pi_2(f_k)>
\sum_{k=j}^{i-1}\pi_2(w_k)-1=B_i-1-B_j$, a contradiction.\\
\end{itemize}

\end{proof}

\begin{corollary}
    There exists an invariant game having $\mathcal{P}$
    as set of $P$-positions if and only if the $4$-tuple
    $(\alpha,\beta,\gamma,\delta)$ is good.
\end{corollary}

\section{Characterizing good $4$-tuples}\label{sec:4}

In the previous section, we have described combinatorial conditions
(expressed in Definition~\ref{def:pbad} and
Definition~\ref{def:pbad2}) on the word $\mathbf{w}_+$ or
$\mathbf{w}_-$ leading to the existence of an invariant game. In this
section, we translate these conditions into an algebraic setting
better suited to tests.

The tests described in this section are all of the following kind.
Take two intervals $I,J$ over $[0,1)$ interpreted as intervals over
the unit circle $\mathbb{T}^1=\mathbb{R}/\mathbb{Z}$, i.e., if $a>b$, then the interval
$[a,b)$ is $[a,1)\cup[0,b)$. For a given $4$-tuple
$(\alpha,\beta,\gamma,\delta)$ of real numbers, we ask, whether or not
there exists some $i$ such that $R_{\alpha,\beta}^i(\gamma,\delta)\in
I\times J$.

Recall that $\alpha,\beta,1$ are {\em rationally independent} (i.e.,
linearly independent over $\mathbb{Q}$), if whenever there exist
integers $p$ and $q$ such that $p\alpha+q\beta$ is an integer, then
$p=q=0$.

The extension of the density theorem of Kronecker is well-known: the set
$$\{ R_{\alpha,\beta}^i(\gamma,\delta)=(\{i\alpha+\gamma\},\{i\beta+\delta\})\in\mathbb{T}^2 \mid i\in\mathbb{N}\}$$
is dense in $\mathbb{T}^2$ if and only if $\alpha,\beta,1$ are
rationally independent \cite{Z}.  So, in that latter case, there exist
infinitely many $i$ such that $R_{\alpha,\beta}^i(\gamma,\delta)$
belongs to a non-empty interval $I\times J$.

If $\alpha,\beta,1$ are rationally dependent, since $\alpha$ and
$\beta$ are irrational numbers, there exist integers $p,q,r$ with
$p,q\neq 0$ such that $p\alpha+q\beta=r$. From \eqref{eq:r1}, we
deduce that $q\beta^2+(p-q-r)\beta+r=0$, i.e., $\beta$ is thus an
algebraic number of degree $2$. Of course, the same conclusion holds
for $\alpha$. In this situation, the set of points
$\{R^{n}_{\alpha,\beta}(\gamma,\delta)\mid n\in\mathbb{N}\}$ is dense
on a straight line in $\mathbb{T}^2$ with rational slope, see for instance
Example~\ref{exa:line}. Hence the initial question is reduced to
determine whether or not a line intersect a rectangle. Moreover, if
$\alpha$ and $\beta$ are irrational numbers satisfying \eqref{eq:r1}
but they are not algebraic of degree $2$, then $\alpha,\beta,1$ are
rationally independent.

\begin{example}\label{exa:line}
    Consider the positive root
    $\alpha=(3+\sqrt{17})/2$ of $x^2 - 3 x - 2$. We get $\beta=(7+\sqrt{17})/8$ and $\alpha=4\beta-2$. Hence, we get
$$R^{n}_{\alpha,\beta}(x,y)=(\{x+n\alpha\},\{y+n\beta\})=(\{x+4n\beta\},\{y+n\beta\})$$
showing that $R_{\alpha,\beta}$ corresponds to a translation of
$\{\beta\}(4,1)$ in $\mathbb{T}^2$. Since $\beta$ is irrational,
thanks to Kronecker theorem, the set of points
$\{R^{n}_{\alpha,\beta}(x,y)\mid n\in\mathbb{N}\}$ is dense on the
straight line in $\mathbb{T}^2$ with rational slope $1/4$ and passing
through $(x,y)$.
\begin{figure}[htbp]  \centering
    \scalebox{1}{\includegraphics{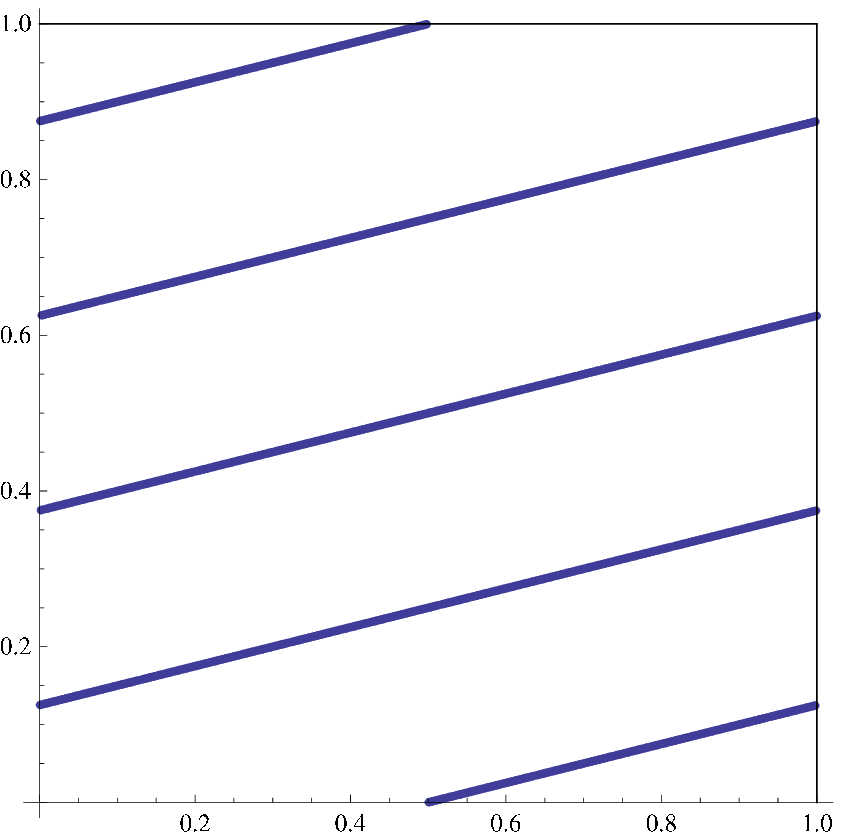}}
    \caption{An example of orbit when $\alpha,\beta,1$ are rationally dependent.}
    \label{fig:dep}
\end{figure}
\end{example}

\subsection{Testing a bad prefix in $\mathbf{w}_+$}
Let $p=w_0u_1\cdots u_n=w_0u$ be a prefix of length $n+1$ occurring in
$\mathbf{w}_+$. If $u_1\cdots u_{n-1}$ is not a palindrome or if $u\not\in\{a,b\}^*$ or if $2B_1\not\in\{\pi_2(u_n)-1,\pi_2(u_n),\pi_2(u_n)+1\}$, then $p$
is not bad. Otherwise, one of the following three situations may occur.

If $2B_1=\pi_2(u_n)-1$, then $p$ is bad if and only if $\pi_1(u)$ is
light, $\pi_2(u)$ is heavy, i.e., $R_{\alpha,\beta}(\gamma,\delta)\in
I_{L,\alpha}(n)\times I_{H,\beta}(n)$ where, as in \eqref{eq:IH},
$I_{L,\alpha}(n)=[0,1-\{n\alpha\})$ and
$I_{H,\alpha}(n)=[1-\{n\beta\},1)$, and there exists a factor $v\in\Fac_n(\mathbf{w}_+)$ such that $\pi_1(v)$ is
heavy, $\pi_2(u)$ is light, i.e., there exists $i$ such that
$$R_{\alpha,\beta}^i(\gamma,\delta)\in I_{H,\alpha}\times
I_{L,\beta}(n).$$

If $2B_1=\pi_2(u_n)$, then $p$ is bad if and only if $\pi_1(u)$ is
light and there exists a factor $v\in\Fac_n(\mathbf{w}_+)$ such that $\pi_1(v)$ is
heavy and both $\pi_2(u)$ and $\pi_2(v)$ are either light or heavy, i.e., there exists $i$ such that
$$R_{\alpha,\beta}^i(\gamma,\delta)\in I_{H,\alpha}\times
I_{L,\beta}(n) \text{ if }R_{\alpha,\beta}(\gamma,\delta)\in
I_{L,\alpha}(n)\times I_{L,\beta}(n),$$
$$R_{\alpha,\beta}^i(\gamma,\delta)\in I_{H,\alpha}\times
I_{H,\beta}(n) \text{ if }R_{\alpha,\beta}(\gamma,\delta)\in
I_{L,\alpha}(n)\times I_{H,\beta}(n).$$

If $2B_1=\pi_2(u_n)+1$, then $p$ is bad if and only if $\pi_1(u)$ and $\pi_2(u)$ are light, i.e., $R_{\alpha,\beta}(\gamma,\delta)\in
I_{L,\alpha}(n)\times I_{L,\beta}(n)$ and there exists a factor $v\in\Fac_n(\mathbf{w}_+)$ such that $\pi_1(v)$ and $\pi_2(u)$ are heavy, i.e., there exists $i$ such that
$$R_{\alpha,\beta}^i(\gamma,\delta)\in I_{H,\alpha}\times
I_{H,\beta}(n).$$
\subsection{Testing (B.1)}
Given $u=u_1\cdots u_n\in \{a,b\}^*\cap\Fac(\mathbf{w}_-)$, we can
proceed as follows to decide whether or not $u$ satisfies (B.1).  If
$\pi_2(u_1\cdots u_n)$ is light, i.e., $|\pi_2(u_1\cdots
u_n)|_{\lfloor\beta\rfloor+1}=\lceil n\{\beta\} \rceil-1$, then $u$ does not satisfy (B.1).
Otherwise, $\pi_2(u_1\cdots u_n)$ is heavy. In that case,  $u$ satisfies (B.1) if and only if,  there exists an
integer $i$ such that
\begin{itemize}
  \item $1^n$ occurs in $\mathbf{s}_{\alpha,\gamma}$ in position $i$ and,
  \item the factor of length $n$ occurring in
    $\mathbf{s}_{\beta,\delta}$ in position $i$ is light.
\end{itemize}
These last two conditions can be tested as follows. As in \eqref{eq:Iv}, consider the two intervals
$$I_{1^n,\alpha}=I_1\cap R_\alpha^{-1}(I_1)\cap\cdots \cap R_\alpha^{-n+1}(I_1) \text{ where }I_1=[0,1-\{\alpha\})$$
and, as in \eqref{eq:IH}, $I_{L,\beta}(n)=[0,1-\{n\beta\})$. Using
Lemma~\ref{lem:algo}, the two above conditions hold true if and only
if $I_{1^n,\alpha}\neq\emptyset$ and there  exists $i$ such that
$$R_{\alpha,\beta}^i(\gamma,\delta)\in I_{1^n,\alpha}\times
I_{L,\beta}(n).$$

\subsection{Testing (B.2)}
Given $u=u_1\cdots u_n\in \{a,b\}^*\cap\Fac(\mathbf{w}_-)$, we can
proceed as follows to decide whether or not $u$ satisfies (B.2).  If
$\pi_2(u_1\cdots u_{n-1})$ is light, i.e., $|\pi_2(u_1\cdots
u_{n-1})|_{\lfloor\beta\rfloor+1}=\lceil (n-1)\{\beta\} \rceil-1$,
then $u$ does not satisfy (B.2).  Otherwise, $\pi_2(u_1\cdots
u_{n-1})$ is heavy.

If $\lceil (n-1)\{\alpha\}\rceil>2$, then any factor of length $n-1$
in $\mathbf{s}_{\alpha,\gamma}$ contains at least two symbols $2$ and
$u$ does not satisfy (B.2).

If $\lceil (n-1)\{\alpha\}\rceil=1$ (resp. if $\lceil
(n-1)\{\alpha\}\rceil=2$), then any factor of length $n-1$ in
$\mathbf{s}_{\alpha,\gamma}$ with exactly one symbol $2$ is heavy (resp. light).  In
that case, $u$ satisfies (B.2) if and only if, there exists an integer
$i$ such that
\begin{itemize}
  \item the factor of length $n-1$ occurring in
    $\mathbf{s}_{\alpha,\gamma}$ in position $i$ is heavy (resp. light),
  \item the factor of length $n-1$ occurring in
    $\mathbf{s}_{\beta,\delta}$ in position $i$ is light.
\end{itemize}
The two above conditions hold true if and only if there exists $i$
such that $$R_{\alpha,\beta}^i(\gamma,\delta)\in
I_{H,\alpha}(n-1)\times I_{L,\beta}(n-1)$$ (resp. if and only if there
exists $i$ such that $R_{\alpha,\beta}^i(\gamma,\delta)\in
I_{L,\alpha}(n-1)\times I_{L,\beta}(n-1)$).

\section{$B_1$-superadditivity is not a necessary condition for invariance}\label{sec:b1}

As mentioned in Section 1.2, the authors of \cite{Lar} conjectured
that a pair $(A_n,B_n)_{n\ge 1}$ of non-homogeneous complementary
Beatty sequences with $A_1=1$ give rise to an invariant game if and
only if the sequence $(B_n)_{>0}$ is $B_1$-superadditive. In what
follows we provide counterexamples to this assumption, meaning that
the notions of a good $4$-tuple given in Definition~\ref{def:good} and
$B1$-superadditivity are not equivalent.

\subsection{Counterexamples with $\gamma<0$} All sequences satisfying
$\gamma<0$ are never superadditive, but some admit invariant games.
The non-superadditivity of such sequences can be easily proved in the
following Lemma.
\begin{lemma}
    Given a pair $(A_n,B_n)_{n\ge 1}=(\lfloor n\alpha+\gamma
    \rfloor,\lfloor n\beta+\delta \rfloor)_{n\ge 1}$ of
    non-homogeneous complementary Beatty sequences with $\gamma<0$,
    there exists some integers $n,m>0$ such that $B_m+B_n>B_{m+n}$.
\end{lemma}
\begin{proof}
First note that according to Remark \ref{rem2}, we have $0<\delta<1$. Since $\beta$ is irrational, there exists some integer $n>0$ such that $\{\beta n +\delta\}<\frac{\delta}{2}$. With such an $n$ we have $$B_n=\lfloor \beta n+\delta \rfloor=\beta n+\delta-\{\beta n +\delta\}>\beta n+\frac{\delta}{2}.$$ By multiplying par $2$ we obtain
$$2B_n>\beta 2n+\delta>B_{2n},$$ showing the desired result.
\end{proof}

Now, all the wanted counterexamples are those satisfying Definition
\ref{def:good}. As an illustration, take for instance the sequence
with
\begin{equation}
    \label{eq:paramneg}
    \beta=1.99+\frac{\sqrt{5}}{2},\quad \alpha=\frac{\beta}{\beta-1},\quad \gamma=-0.2\ \text{ and }\ \delta=-\frac{\beta\gamma}{\alpha}.
\end{equation}

$$\begin{array}{c|cccccccccccccccc}
n&0& 1& 2& 3& 4& 5& 6& 7& 8& 9& 10& 11& 12& 13& 14\\
\hline
A_{n+1}-A_n&1& 1& 2& 1& 2& 1& 2& 1& 2& 1& 2& 1& 1& 2& 1 \\
B_{n+1}-B_n&3& 3& 3& 3& 3& 4& 3& 3& 3& 3& 3& 3& 3& 3& 4\\
\hline
\mathbf{w}_-&a&a&c&a&c&b&c&a&c&a&c&a&a&c&b\\
\end{array}$$

This sequence admits an invariant game since the $4$-tuple $(\alpha,\beta,\delta,\gamma)$ is good. Indeed, according to Definition \ref{def:good} and since $B_1=3$, it suffices to show that the prefixes $a$ and $aa$ are not suffix-bad. Clearly, since they do not contain any letter $b$, they satisfy neither $(B.1)$ nor $(B.2)$ of Definition \ref{def:pbad2}.

\subsection{Counterexamples with $\gamma>0$}
In the case where $\gamma>0$, one can also find sequences that are
superadditive, not $B_1$-superadditive, and that admit an invariant
game. As an example, consider the following real numbers
\begin{equation}
    \label{eq:param}
    \beta=8+\frac{1+\sqrt{5}}{2},\quad \alpha=\frac{\beta}{\beta-1},\quad \delta=-\frac{5\sqrt{7}}{2}\ \text{ and }\ \gamma=-\frac{\delta\alpha}{\beta}.
\end{equation}
We give below the first terms of the corresponding two sequences
    $$\begin{array}{c|ccccccccccccccc}
n&1& 2& 3& 4& 5& 6& 7& 8& 9& 10& 11& 12& 13& 14& 15\\
\hline
A_n&1& 2& 4& 5& 6& 7& 8& 9& 10& 11& 13& 14& 15& 16& 17\\
B_n&3& 12& 22& 31& 41& 51& 60& 70& 79& 89& 99& 108& 118& 128& 137\\
\end{array}$$

We have $B_{1+2}=22>B_1+B_2+B_1$ showing that $(B_n)_{n\ge 1}$ is not
$B_1$-superadditive. However this sequence is superadditive as it is
the case for all sequences having $\gamma>0$.  It now remains to show
that with the parameters given by \eqref{eq:param}, the sequence
$(A_n,B_n)$ corresponds to the $P$-positions of an invariant game.
For this purpose, it suffices to make use of Theorem~\ref{the:principal} and
detect directly whether or not the $4$-tuple
$(\alpha,\beta,\gamma,\delta)$ is good. This can be easily carried on as follows.
   $$\begin{array}{c|cccccccccccccccc}
n&0& 1& 2& 3& 4& 5& 6& 7& 8& 9& 10& 11& 12& 13& 14\\
\hline
A_{n+1}-A_n&1& 1& 2& 1& 1& 1& 1& 1& 1& 1& 2& 1& 1& 1& 1 \\
B_{n+1}-B_n&3& 9& 10& 9& 10& 10& 9& 10& 9& 10& 10& 9& 10& 10& 9\\
\hline
\mathbf{w}_+&e&a&d&a&b&b&a&b&a&b&d&a&b&b&a\\
\end{array}$$
From Definition~\ref{def:good}, since $B_1=3$, the only prefix of $\mathbf{w}_+$ to test is $p=ea$. If this prefix is not bad, then the $4$-tuple $(\alpha,\beta,\gamma,\delta)$ is good and from Theorem~\ref{the:principal}, the existence of an invariant game is guaranteed. In particular, $a$ is a palindrome. But $2B_1=6$ and $\pi_2(a)=9$. Hence $2B_1\not\in\{8,9,10\}$ and we can conclude directly that the prefix is good.

\begin{remark}
    Note that for the sequences satisfying $B_1=2$, the current paper
    does not provide any characterization for admitting an invariant
    game. Hence the result of Larsson {\em et al.} \cite{Lar} remains
    the best one for such sequences, asserting that
    $B_1$-superadditivity is a sufficient condition for having an
    invariant game. Yet, it is not a necessary condition any more in
    that context, since there also exist counterexamples of non
    $B1$-superadditive sequences that correspond to $P$-positions of
    invariant games.
\end{remark}

\end{document}